\documentclass[12pt,leqno,twoside]{article}
\usepackage{amssymb}
\usepackage{amsmath}
\usepackage{amsthm}
\usepackage{t1enc}
\usepackage[cp1250]{inputenc}
\usepackage{a4,indentfirst,latexsym}
\usepackage{graphics}
\usepackage{mathrsfs}
\usepackage{cite,enumitem,graphicx}
\usepackage[colorlinks=true,urlcolor=black,
citecolor=black,linkcolor=black,linktocpage,pdfpagelabels,
bookmarksnumbered,bookmarksopen]{hyperref}
\usepackage[colorinlistoftodos]{todonotes}

\input xy
\xyoption{all}

\newtheorem{Th}{Theorem}[section]
\newtheorem{Prop}[Th]{Proposition}
\newtheorem{Lem}[Th]{Lemma}
\newtheorem{Cor}[Th]{Corollary}

\newenvironment{altproof}[1]
{\noindent
{\em Proof of {#1}}.}
{\nopagebreak\mbox{}\hfill $\Box$\par\addvspace{0.5cm}}

   \newcommand{\vp}{\varphi}
   
   \newcommand{\eps}{\varepsilon}

   \def\div{\mathop{\mathrm{div}\,}}

   \def\supp{\mathrm{supp}}

   \def\id{\mathrm{id}}

   \def\Z{\mathbb{Z}}

   \def\N{\mathbb{N}}
   \def\R{\mathbb{R}}

   \def\curl{\mathrm{curl}}

   \def\cl{\mathrm{cl\,}}

   \def\U{\mathcal{U}} 
   
   \def\V{\mathcal{V}}
   \def\E{\mathcal{E}}
   \def\J{\mathcal{J}}

   \def\W{\mathcal{W}}
   \def\D{\mathcal{D}}

   \def\A{\mathcal{A}}
     \def\M{\mathcal{M}}

\newcommand{\cB}{{\mathcal B}}
\newcommand{\cC}{{\mathcal C}}
\newcommand{\cD}{{\mathcal D}}
\newcommand{\cE}{{\mathcal E}}

\newcommand{\cH}{{\mathcal H}}
\newcommand{\cI}{{\mathcal I}}

\newcommand{\cL}{{\mathcal L}}
\newcommand{\cM}{{\mathcal M}}
\newcommand{\cN}{{\mathcal N}}

\newcommand{\cP}{{\mathcal P}}

\newcommand{\al}{\alpha}

\newcommand{\Ga}{\Gamma}
\newcommand{\Om}{\Omega}

\def\curlop{\nabla\times}
\newcommand{\weakto}{\rightharpoonup}
\newcommand{\pa}{\partial}
\def\id{\mathrm{id}}

\newcommand{\tX}{\widetilde{X}}

\numberwithin{equation}{section}

\begin{document}

\title{Ground states of time-harmonic 
semilinear Maxwell equations in $\R^3$ with vanishing permittivity}

\author{Jaros\l aw Mederski\footnote{The study was
supported by research grant NCN 2013/09/B/ST1/01963}}
\date{}

\maketitle

\begin{abstract}
We investigate the existence of solutions $E:\mathbb{R}^3\to\mathbb{R}^3$ of the time-harmonic semilinear Maxwell equation 
$$
\curlop(\curlop E) + V(x) E = \pa_E F(x,E) \quad \text{in }\mathbb{R}^3
$$
where $V:\mathbb{R}^3\to\mathbb{R}$, $V(x)\leq 0$ a.e. on $\mathbb{R}^3$,
$\curlop$ denotes the curl operator in $\R^3$ and $F:\mathbb{R}^3\times\mathbb{R}^3\to\mathbb{R}$ is a nonlinear function in $E$. In particular we find a ground state solution provided that suitable growth conditions on $F$ are imposed and  $L^{3/2}$-norm of $V$ is less than the best Sobolev constant. In applications $F$ is responsible for the nonlinear polarization and $V(x)=-\mu\omega^2\varepsilon(x)$ where $\mu>0$ is the magnetic permeability, $\omega$ is the frequency of the time-harmonic electric field $\Re\{E(x)e^{i\omega t}\}$ and $\varepsilon$ is the linear part of the permittivity in an inhomogeneous medium. 
\end{abstract}



\section*{Introduction}
\setcounter{section}{1}

We study the propagation of electromagnetic waves $(\cE,\cB)$ in the absence of charges, currents and magnetization. The constitutive relations between the electric displacement field $\cD$ and the electric field $\cE$ as well as between the magnetic induction $\cH$ and the magnetic field $\cB$ are given by 
\begin{equation}\label{eq:relations}
\cD=\eps\cE +\cP_{NL} \quad\text{and}\quad \cH=\frac1\mu \cB,
\end{equation}
where $\eps$ is the (linear) permittivity of an inhomogeneous material, and $\cP_{NL}$
stands for the nonlinear polarization which depends nonlinearly on the electric field $\cE$. In inhomogeneous media $\eps$ and $\cP_{NL}$ depend on the position $x\in\R^3$ and  we assume that the magnetic permeability is constant $\mu>0$.
As usual, the Maxwell equations
\begin{equation}\label{eq:Maxwell}
\left\{
\begin{aligned}
    &\curlop \cH = \pa_t \cD, \quad \div(\cD)=0,\\
    &\pa_t \cB + \curlop \cE=0, \quad\div(\cB)=0,
\end{aligned}
\right.
\end{equation}
together with the constitutive relations \eqref{eq:relations} lead to the equation (see Saleh and Teich \cite{FundPhotonics})
$$
\curlop\left(\frac1\mu\curlop \cE\right)+\partial_t^2 (\eps\cE)
 = -\partial_t^2 \cP_{NL}.
$$
In the time-harmonic case the fields $\cE$ and $\cP_{NL}$ are of the form
$\cE(x,t) = \Re\{E(x)e^{i\omega t}\}$, $\cP_{NL}(x,t) = \Re\{P(x)e^{i\omega t}\}$, where $E(x),P(x)\in\R^3$ and we arrive at the time-harmonic Maxwell equation
\begin{equation}\label{eq}
\nabla\times(\nabla\times E) + V(x) E = f(x,E) \qquad\textnormal{ in } \R^3,
\end{equation}
where $V(x)=-\mu\omega^2\eps(x)\leq 0$ and $f(x,E)=\mu\omega^2 P(x,E)$. Here $E:\R^3\to\R^3$ is a vector field and $V:\R^3\to\R$. In a Kerr-like medium the strong electric field $\cE$ of high intensity causes the refractive index to vary quadratically with the field and then the polarization has the form $\cP_{NL} = \al(x)\langle |\cE|^2\rangle\cE$, where  $\langle |\cE|^2\rangle$ stands for the time average of the intensity of $\cE$, hence $P(x,E)=\frac{1}{2}\al(x)|E|^2E$ (see Nie \cite{Nie} and Stuart \cite{Stuart91}).  In applications, for low intensity $|\cE|$ the Kerr effect is often considered to be linear, $\cP_{NL}$ is negligible and therefore we may assume that $\cP_{NL}$ decays rapidly as $|\cE|\to 0$. In order to model these nonlinear phenomena we consider nonlinearities of the form
\begin{equation}\label{ExKerr}
f(x,E)=\Gamma(x)\min\{|E|^{p-2},|E|^{q-2}\}E,\quad 2<p\leq q, 
\end{equation}
where $\Gamma\in L^{\infty}(\R^3)$ is positive, periodic and bounded away from $0$. Case $p=4$ corresponds to the Kerr effect for the strong field $\cE$. In fact, we will able to deal with general nonlinearities of the form $f(x,E)=\partial_E F(x,E)$, where $F:\R^3\times\R^3\to\R$. Some other examples of nonlinearities in physical models can be found e.g. in Stuart \cite{Stuart91} (see also Section \ref{sec:results}).


We look for weak solutions to \eqref{eq} in a certain $\D(\curl,p,q)$ space, where $p$ and $q$ are provided by the growth of $f$; see Section \ref{sec:varsetting} for details. Note that a solution $E$ of \eqref{eq}
determines $\cP_{NL}$ and $\cD$ by the first constitutive relation in \eqref{eq:relations} whereas $\cB$ and $\cH$ are obtained from $\nabla\times \cE$ by time-integration. We will show that if $E\in \D(\curl,p,q)$ solves \eqref{eq}, then the total electromagnetic energy
\begin{eqnarray}\label{eq:EM_Energy}
\cL(t):=\frac12\int_{\R^3}\cE\cD+\cB\cH\,dx
\end{eqnarray}
is finite. We do not know whether the fields $\cE$, $\D$, $\cB$ and $\cH$ are localized, i.e. decay to zero as $|x|\to\infty$, however $\D(\curl,p,q)$ lies in the sum of Lebesgue spaces $L^{p,q}:=L^p(\R^3,\R^3)+L^q(\R^3,\R^3)$ and therefore it does not contain the usual nontrivial travelling waves $E$ propagating in a given direction $z\in\R^3$ such that $E(x)=E(x+z)$ for all $x\in\R^3$. The finiteness of the electromagnetic energy and the localization problem attract a strong attention in the study of self-guided beams of light in a nonlinear medium; see e.g. \cite{Stuart91,Stuart93}. 

We restrict our considerations to optical metama\-terials having permittivity $\eps$ close to zero, i.e. the so-called epsilon-near-zero (ENZ) media (see e.g. \cite{BoostingPhysRevB,Linearpermittivity,NonlinearPlasmonics} and references therein). The ENZ materials exhibit strong nonlinear effects, e.g. the Kerr effect, governed by the  polarization $\cP_{NL}$ and the propagation of time-harmonic electric field waves is described by \eqref{eq}. 
Our principal aim is to investigate the existence and the nonexistence of solutions to \eqref{eq} under appropriate assumptions imposed on $V$ and $F$. In particular, the closeness to zero of $\eps$ will be expressed in terms of $L^{\frac32}$-norm of $V$ (see Section 2). Moreover ground state solutions which have the least possible energy among all nontrivial solutions will be of our major interest owing to their physical importance. It is worth mentioning that usually naturally occurring materials have the permittivity positive and bounded away from zero, i.e. $V(x)=-\mu\omega^2\eps(x)$ is negative and bounded away from $0$. However it is not clear in which space one should seek weak solutions of this problem with such $V$ and a nonlinearity of the form \eqref{ExKerr}, and whether any variational method can be used. We will show, in fact, that \eqref{eq} does not admit classical solutions in case of constant and negative $V$; see Corollary \ref{CorollaryMain}.

Recall that semilinear equations involving the the curl-curl operator $\curlop \curlop(\cdot)$ in $\R^3$ have been recently studied by Benci and Fortunato in \cite{BenFor}. They introduce a model for a unified field theory for classical electrodynamics which is based on a semilinear perturbation of the Maxwell equations. In the magnetostatic case, in which the electric field vanishes and the magnetic field is independent of time, they are lead to an equation of the form
\begin{equation}\label{eq:benci-fortunato}
\curlop(\curlop A) = W'(|A|^2)A\qquad\text{in } \R^3
\end{equation}
for the gauge potential $A$ related to the magnetic field $H=\curlop A$. Here
$F(A)=\frac12 W(|A|^2)$ is nonlinear in $A$.
We emphasize that proof of the existence of solutions to (\ref{eq:benci-fortunato}) in \cite{BenFor} contains a gap 
and the techniques from \cite{BenFor} do not seem to be sufficient. Indeed, in order to deal with the lack of compactness issue they restrict the space of divergence-free vector fields to the radially symmetric ones, which becomes the null space.  Finally in \cite{BenForAzzAprile} Azzollini et al.\ use the cylindrical symmetry of the equation to find solutions of \eqref{eq:benci-fortunato} of the form
\begin{equation*}\label{eq:sym1}
A(x)=\al(r,x_3)\begin{pmatrix}-x_2\\x_1\\0\end{pmatrix},\qquad r=\sqrt{x_1^2+x_2^2}.
\end{equation*}
A field of this form is divergence-free and 
$$\curlop \curlop A=-\Delta A,$$
hence standard methods of nonlinear analysis apply. In \cite{DAprileSiciliano} D'Aprile and Siciliano find another kind of cylindrical solutions of the equation again using symmetry arguments and the scaling properties of \eqref{eq:benci-fortunato}.
Observe that \eqref{eq} cannot be treated neither by the Palais principle of symmetric criticality \cite{Palais} nor by the rescaling arguments due to the presence of nonsymmetric and vanishing $V$, i.e. $V\in L^{\frac32}(\R^3)$. We would like to emphasize that we are also able to deal with functions $F(x,E)$ that depend on $x$ and are not radial in $E$. Therefore, we point out that the existence of ground states solutions of \eqref{eq} with $V=0$ will shed a new light on equation (\ref{eq:benci-fortunato}) and on a new formulation of the Maxwell equations due to Born and Infeld \cite{BornInfeld,BenFor}.

We also mention the papers \cite{Stuart91,Stuart93,StuartZhou96,StuartZhou05,StuartZhou10,StuartZhou01,Stuart04} by Stuart and Zhou, who studied transverse electric and transverse magnetic solutions to \eqref{eq:Maxwell} for asymptotically linear polarizations and if again the cylindrical symmetry is present.

Problem \eqref{eq} has a variational structure and (weak) solutions correspond to critical points of the energy functional
\begin{equation}\label{eq:action}
\E(E) = \frac12\int_{\R^3}|\curlop E|^2\,dx + \frac12\int_{\R^3} V(x)|E|^2\,dx - \int_{\R^3} F(x,E)\,dx
\end{equation}
defined on a space $\D(\curl,p,q)$ which will be introduced in Section~3. One difficulty from a mathematical point of view is that the curl-curl operator $\curlop \curlop(\cdot)$ has an infinite-dimensional kernel, namely all gradient vector fields. Moreover the functional $\E$ is unbounded from above and from below and its critical points have infinite Morse index. In addition to these problems related to the strongly indefinite geometry of $\E$, we also have to deal with the lack of compactness issues. Namely functional $\E'$ is not (sequentially) weak-to-weak$^*$ continuous, i.e. the weak convergence $E_n\weakto E$ in $\D(\curl,p,q)$ does not imply that $\E'(E_n)\weakto \E'(E)$ in $\D(\curl,p,q)^*$ (see the discussion preceding Corollary \ref{CorJweaklycont}). Therefore we do not know whether a weak limit of a bounded Palais-Smale sequence is a critical point. Moreover the lack of the sufficient regularity of $\E$ makes this problem difficult to treat with the available variational 
methods for indefinite problems e.g. \cite{BenciRabinowitz,BartschDing,Rabinowitz:1986}.

In order to find solutions to \eqref{eq} we use a generalization of the Nehari manifold technique for strongly indefinite functionals obtained recently by Bartsch and the author in \cite{BartschMederski} (see also Szulkin and Weth \cite{SzulkinWeth,SzulkinWethHandbook}). Namely we introduce a Nehari-Pankov manifold (cf. \cite{Pankov}) which is homeomorphic with a sphere in the subspace of $\D(\curl,p,q)$ consisting of divergence-free vector fields. This allows to find a minimizing sequence on the sphere and hence on the Nehari-Pankov manifold. However in \cite{BartschMederski} we are in a position to find a limit point of the sequence being a critical point because the space of divergence-free vector fields on a bounded domain is compactly embedded into certain $L^p$ spaces and a variant of the Palais-Smale condition is satisfied. Since \eqref{eq} is modelled in $\R^3$, the minimizing sequences are no longer compact. Therefore critical point theory developed in \cite{BartschMederski}[Section 4] is 
insufficient to find a solution to \eqref{eq}. Moreover the lack of the weak-to-weak$^*$ continuity of $\cE'$ makes this problem impossible to treat by a concentration-compactness argument in the spirit of Lions \cite{Lions84,Lions85} in $\D(\curl,p,q)$.
Our approach is based on a new careful analysis of a bounded sequence $(E_n)$ of the Nehari-Pankov manifold (Theorem \ref{ThMainSplitting}) with a possibly infinite splitting \eqref{eqInfinitesplitting} of the limit
$$\lim_{n\to\infty}\Big(\frac{1}{2}\int_{\R^3}|\curlop E_n|^2\,dx-\E(E_n)\Big).$$
This result enables us to get the the weak-to-weak$^*$ continuity of $\cE'$ on the Nehari-Pankov manifold. Moreover, in the spirit of the global compactness result of Struwe \cite{Struwe,StruweSplitting} or Coti Zelati and Rabinowitz \cite{CotiZelatiRab}, we are able to find a finite splitting of the ground state level $\lim_{n\to\infty}\cE(E_n)$ with respect to a minimizing sequence $(E_n)$ of the Nehari-Pankov manifold (Theorem  \ref{ThMainPS_Splitting}). Finally comparisons of energy levels will imply the existence of solutions to \eqref{eq} (Theorem \ref{ThMain}).

The paper is organized as follows. In Section 2 we formulate our hypotheses on $V$ and $F$, and we state our main results concerning the existence and the nonexistence of solutions and ground state solutions. In Section 3 we introduce the variational setting, in particular the spaces on which $\E$ will be defined.
Moreover we provide the Helmholtz decomposition of a vector field $E$ into the divergence-free component $u$ and the curl-free component $\nabla w$, that allows to treat $\E$ as a functional $\J$ of two variables $(u,w)$ (see \eqref{eqJ} and Proposition \ref{PropSolutE}). 
Next, in Section 4 we  introduce the Nehari-Pankov manifold on which we minimize $\J$ in order to find a ground state. In Section 5 we provide an analysis of bounded sequences in $\D(\R^3,\R^3)$ and we obtain a splitting of a bounded sequence of the Nehari-Pankov manifold in Theorem~\ref{ThMainSplitting}. 
We investigate Palais-Smale sequences  in Section 6 and we prove Theorem \ref{ThMainPS_Splitting}. Finally in Section 7 we prove Theorem \ref{ThMain} which states the existence of solutions and ground state solutions of \eqref{eq} and we obtain a variational identity in Theorem \ref{ThPohozaev} implying a nonexistence result Corollary \ref{CorollaryMain}.

\section{Main results}\label{sec:results}

We impose on $V:\R^3\to\R$ the following condition.
\begin{itemize}
\item[(V)] $V\in L^{\frac{p}{p-2}}(\R^3)\cap L^{\frac{q}{q-2}}(\R^3)$, $V\leq 0$ a.e. on $\R^3$ and $|V|_{\frac{3}{2}}<S$, where
$$S:=\inf_{u\in \D^{1,2}\setminus\{0\}}\frac{\int_{\R^3}|\nabla u|^2\,dx}{
|u|_6^2}$$
is the classical best Sobolev constant.
\end{itemize}
Here and in the sequel $|\cdot|_q$ denotes the $L^q$-norm.
Now we collect assumptions on the nonli\-nearity $F(x,u)$.
\begin{itemize}
\item[(F1)] $F:\R^3\times\R^3\to\R$ is differentiable with respect to the second variable $u\in\R^3$, and $f=\pa_uF:\R^3\times\R^3\to\R^3$ is a Carath\'eodory function (i.e.\ measurable in $x\in\R^3$, continuous in $u\in\R^3$ for a.e.\ $x\in\R^3$). Moreover $f$ is $\Z^3$-periodic in $x$ i.e. $f(x,u)=f(x+y,u)$ for $x,u\in \R^3$ and $y\in\Z^3$.

\item[(F2)] If $V< 0$ a.e. on $\R^3$ then $F$ is convex in $u\in\R^3$, otherwise $F$ is uniformly strictly convex with respect to $u\in\R^3$, i.e.\ for any compact $A\subset(\R^3\times\R^3)\setminus\{(u,u):\;u\in\R^3\}$
$$
\inf_{\genfrac{}{}{0pt}{}{x\in\R^3}{(u_1,u_2)\in A}}
 \left(\frac12\big(F(x,u_1)+F(x,u_2)\big)-F\left(x,\frac{u_1+u_2}{2}\right)\right) > 0.
$$

\item[(F3)] There are $2<p<6<q$ and constants $c_1,c_2>0$ such that
$$F(x,u)\geq c_1 \min(|u|^{p},|u|^{q})$$
and 
$$|f(x,u)|\leq c_2 \min(|u|^{p-1},|u|^{q-1})$$
for all $x,u\in\R^3$.

\item[(F4)] For any $x\in\R^3$ and
$u\in\R^3$, $u\neq 0$
$$\langle f(x,u), u\rangle > 2 F(x,u).$$

\item[(F5)] If $ \langle f(x,u),v\rangle = \langle f(x,v),u\rangle \ne 0\ $ then
$\ \displaystyle F(x,u) - F(x,v)
 \le \frac{\langle f(x,u),u\rangle^2-\langle f(x,u),v\rangle^2}{2\langle f(x,u),u\rangle}$.\\
If in addition $F(x,u)\ne F(x,v)$ then the strict inequality holds.
\end{itemize}

The periodicity arises in the study of dielectric materials, e.g. in photonic crystals and we assume it in (F1). The convexity condition (F2) is rather harmless (see examples below) and observe that condition (F4) is reminiscent of the Ambro\-setti-Rabinowitz condition. The growth condition (F3) describes a supercritical behavior $|u|^q$ of $F$ for $|u|$ small and subcritical behavior $|u|^p$ for large $|u|$. Note that $6=2^*$ is the critical Sobolev exponent. This kind of growth has been considered for Schr\"odinger equations in the zero-mass case e.g. by Berestycki and Lions \cite{BerLions} or Benci, Grisanti and Micheletti \cite{BenGrisantiMich}. Moreover, similarly as in \cite{BenFor,BenForAzzAprile,DAprileSiciliano} in the study of \eqref{eq:benci-fortunato}, condition (F3) requires to work in $L^{p,q}$ in order to ensure that the nonlinear term of energy functional \eqref{eq:action} is finite; see Section \ref{sec:varsetting} for details.
The technical condition (F5) is a variant of the monotonicity condition for vector fields (see e.g. Szulkin and Weth \cite{SzulkinWeth})  and will be needed to set up the  Nehari-Pankov manifold (cf. conditions (F1) - (F7) in \cite{BartschMederski}).

Our model examples are of the form 
\begin{eqnarray}\label{Example1}
F(x,u)&=&\left\{
\begin{array}{ll}
    \Ga(x)\big(\frac1p|Mu|^p+\frac1q-\frac1p\big)
    &
    \hbox{if } |Mu|>1,\\
    \Ga(x) \frac1q|Mu|^q
    &
    \hbox{if } |Mu|\leq 1,\\
\end{array}
\right.\\
\label{Example2}
F(x,u)&=&\Ga(x)\frac1p\big((1+|Mu|^q)^{\frac{p}{q}}-1\big)
\end{eqnarray}
with $\Ga\in L^\infty(\R^3)$ is $\Z^3$ periodic, positive and bounded away from $0$, $M\in GL(3)$ is an invertible $3\times 3$ matrix, $2<p<6<q$.  Then all assumptions on $F$ are satisfied. Observe that these functions are not radial when $M$ is not an orthogonal matrix. Of course, if $M=\id$, then for \eqref{Example1},  $f(x,u)$ takes the form of \eqref{ExKerr}.  Other examples can be provided by considering radial functions of the form $F(x,u)=W(|u|^2)$, where $W\in \cC^1(\R,\R)$, $W(0)=W'(0)=0$ and $W'(t)$ is strictly increasing on $(0,+\infty)$. Then we check that (F1), (F2), (F4) and (F5) are satisfied.

\indent Our principal aim is to prove the following result.
\begin{Th}\label{ThMain} Assume that (F1)-(F5) and (V) hold. Then there is a solution to \eqref{eq}. If  $V<0$ a.e. on $\R^3$ or $V=0$ then \eqref{eq} has a ground state solution, i.e. there is a critical point $E\in\M$ of $\E$ such that
$$\E(E)=\inf_{\mathcal{M}}\E>0,$$
where
\begin{eqnarray}\label{DefOfNehari1}
\mathcal{M} &:=& \{E\in \D(\curl,p,q)|\; E\neq 0,\; 
\E'(E)(E)=0,\\\nonumber
&&\hbox{ and }\E'(E)(\nabla\vp)=0\,\hbox{ for any }\vp\in \cC_0^{\infty}(\R^3)\}.
\end{eqnarray}
\end{Th}

Since $\mathcal{M}$ contains all nontrivial critical points of $\E$, then a ground state solution is a nontrivial solution with the least possible energy $\E$. Moreover we show that any $E\in\cM$ admits the Helmholtz decomposition $E=u+\nabla w$ with $u\neq 0$ and $\div(u)=0$.

We provide a careful analysis of bounded sequences in $\cM$ which plays a crucial role in proof of Theorem \ref{ThMain}. Namely, setting
\begin{equation}\label{DefOfI}
I(E):=\frac{1}{2}\int_{\R^3}|\curlop E|^2\,dx-\E(E)= -\frac12\int_{\R^3} V(x)|E|^2\,dx + \int_{\R^3} F(x,E)\,dx 
\end{equation}
we get the following result.

\begin{Th}\label{ThMainSplitting}  Assume that (F1)-(F5) and (V) hold.
If $(E_n)_{n=0}^{\infty}\subset \cM$ is bounded then, up to a subsequence,
there is $N\in \N\cup \{\infty\}$, $\bar{E}_0\in \D(\curl,p,q)$ and there are sequences $(\bar{E}_i)_{i=1}^N\subset  \D(\curl,p,q)\setminus\{0\}$ and $(x_n^i)_{n\geq i}\subset \Z^3$ with $x_n^0=0$ such that the following conditions hold:
\begin{equation}\label{EqThMainSplitting1}
E_n(\cdot+x_n^i)\rightharpoonup \bar{E}_i\hbox{ in }\D(\curl,p,q)\hbox{ and }E_n(\cdot+x_n^i)\to \bar{E}_i\hbox{ a.e. in }\R^3\hbox{ as }n\to\infty, 
\end{equation}
for any $0\leq i < N+1$, and
\begin{equation}\label{EqThMainSplitting2}
E_n-\sum_{i=0}^{\min\{n,N\}}\bar{E}_i(\cdot-x_n^i)\to 0\hbox{ in }L^{p,q}=L^p(\R^3,\R^3)+L^q(\R^3,\R^3)\hbox{ as }n\to\infty.
\end{equation}
Moreover
\begin{equation}\label{eqInfinitesplitting}
\lim_{n\to\infty}I(E_n)=I(\bar{E}_0)+
\sum_{i=1}^N I_0(\bar{E}_i)<\infty,
\end{equation}
where $\cM_0$ and $I_0$ are given by $(\ref{DefOfNehari1})$ and $(\ref{DefOfI})$ under assumption $V=0$.
\end{Th}

As a consequence of Theorem \ref{ThMainSplitting} we get the sequentially weak-to-weak$^*$ continuity of $\E'$ in $\cM\cup\{0\}$ (cf. Corollary \ref{CorJweaklycont}). Moreover, in the spirit of the global compactness result of Struwe \cite{Struwe,StruweSplitting} or Coti Zelati and Rabinowitz \cite{CotiZelatiRab}, we obtain a finite splitting of energy levels with respect to a Palais-Smale sequence in $\cM$.

\begin{Th}\label{ThMainPS_Splitting}  Assume that (F1)-(F5) and (V) hold.
If $(E_n)_{n=0}^{\infty}\subset \cM$ is a $(PS)_c$-sequence at level $c>0$, i.e. $\cE(E_n)\to c$ and $\cE'(E_n)\to 0$, then, up to a subsequence,
there is $\bar{E}_0\in \D(\curl,p,q)$ and a finite sequence $(\bar{E}_i)_{i=1}^N\subset \cM_0$ of critical points of $\E_0$ such that \eqref{EqThMainSplitting1}, \eqref{EqThMainSplitting2}  hold and
\begin{equation}\label{eqPSsplitting}
c=\cE(\bar{E}_0)+
\sum_{i=1}^N \cE_0(\bar{E}_i),
\end{equation}
where $\E_0$ is the energy functional given by \eqref{eq:action} under assumption $V=0$.
\end{Th}

Observe that if $0<c<\inf_{\cM_0} \J_0$ then $N=0$, $\J(\bar{E}_0)=c$ and $\bar{E}_0$ is a nontrivial critical point of $\J$. In this way the comparison of energy levels will imply the existence of nontrivial solutions.

Finally we provide a  variational identity for an autonomous version of \eqref{eq} and we get a corollary justifying to some extent the optimality of growth condition (F3).

\begin{Th}\label{ThPohozaev}
Suppose that $V=0$, $F$ is independent of $x$ and satisfies (F1). 
If $E=u+\nabla w$ is a classical solution to \eqref{eq} such that $\div(u)=0$,
\begin{equation}\label{EqThPohozaev1}
u\in \cC^2(\R^3,\R^3)\hbox{, }w\in \cC^2(\R^3)
\end{equation}
and 
\begin{equation}\label{EqThPohozaev2}
F(E),\langle f(E),\nabla w\rangle\hbox{ and }|f(E)||w|\in L^1(\R^3),
\end{equation}
then
\begin{equation}\label{EqThPohozaev}
\int_{\R^3}|\curlop E|^2\, dx = 6 \int_{\R^3} F(E)\,dx.
\end{equation}
\end{Th}

Observe that for any $2<p\leq q$ the following growth condition 
\begin{itemize}
\item[(F6)] For any 
$x\in\R^3$ and $u\in\R^3$, $u\neq 0$
$$q F(x,u)\geq \langle f(x,u), u\rangle \geq p F(x,u)> 0$$
\end{itemize}
is satisfied by nonlinearities given by \eqref{Example1}, \eqref{Example2} 
and implies the first inequality in (F3).
Now we formulate nonexistence results as a consequence of Theorem \ref{ThPohozaev}. 
\begin{Cor}\label{CorollaryMain} 
Suppose that $F$ is independent of $x$, (F1) and (F6) hold. \\
(a) If  $V=0$, and $2<p\leq q< 6$ or $6< p\leq q$, then there is no classical solution to \eqref{eq} of the form $E=u+\nabla w$ with $u\neq 0$, $\div(u)=0$ satisfying \eqref{EqThPohozaev1} and \eqref{EqThPohozaev2}.\\
(b) If  $V$ is constant and negative,  $2<p\leq q\leq 6$, then there is no classical solution to \eqref{eq} of the form $E=u+\nabla w$ with $u\neq 0$, $\div(u)=0$ satisfying \eqref{EqThPohozaev1}, \eqref{EqThPohozaev2} and
$u\in L^2(\R^3,\R^3)$, $ w\in H^1(\R^3)$.
\end{Cor}

In particular, for the Kerr nonlinearity, i.e. $p=q=4$ and $f(x,E)=|E|^{2}E$ there exist no classical solutions to \eqref{eq} for constant $V\leq 0$. Therefore example \eqref{ExKerr} with $p=4$ and $q>6$ incorporates the Kerr effect only for strong fields $\cE$ in order to solve \eqref{eq}.

\section{Variational setting}\label{sec:varsetting}

Let $1<p\leq q$ and
$$L^{p,q}:=L^p(\R^3,\R^3)+L^q(\R^3,\R^3)$$
denote the Banach space of vector fields $E=E_1+E_2$, where $E_1\in L^p(\R^3,\R^3)$ and $E_2\in L^q(\R^3,\R^3)$,
endowed with the following norm
$$|E|_{p,q}=\sup\left\{\frac{\int_{\R^3}\langle E, F\rangle\, dx}{|F|_{\frac{p}{p-1}}+|F|_{\frac{q}{q-1}}}\Big|\;F\in L^{\frac{p}{p-1}}(\R^3,\R^3)\cap L^{\frac{q}{q-1}}(\R^3,\R^3), F\neq 0\right\}.$$
Recall that in $L^{p,q}$ we can introduce an equivalent norm
$$|E|_{p,q,1}:=\inf\{|E_1|_{p}+|E_2|_q|\;E=E_1+E_2,\;E_1\in L^{p}(\R^3,\R^3),\; E_2\in L^{q}(\R^3,\R^3)\}$$
and by \cite{BadialePisaniRolando}[Proposition 2.5] the infimum in $|\cdot|_{p,q,1}$ is attained. Below we recall some properties of $L^{p,q}$ given e.g. in \cite{BadialePisaniRolando}[Corollary 2.19, Proposition 2.21].
\begin{Lem}\label{LemBenFor}$\hbox{}$\\
$(a)$ If $E\in L^{p,q}$, then
$$\max\Big\{\frac{1}{2}|E\chi_{\Omega_E^c}|_{q}-\frac{1}{2},
\frac{1}{1+|\Omega_E|^{\frac{1}{p}-\frac{1}{q}}}|E\chi_{\Omega_E}|_{p}\Big\}
\leq |E|_{p,q}\leq \max\{|E\chi_{\Omega_E^c}|_{q},
|E\chi_{\Omega_E}|_{p}\},$$
where $\chi_{(\cdot)}$ denotes the characteristic function and $$\Omega_E=\{x\in\R^3|\; |E(x)|> 1\}.$$
$(b)$ A sequence $\{E_n\}\subset L^{p,q}$ is bounded if and only if sequences 
$\{|\Omega_{E_n}|\}$, $\{|E_n\chi_{\Omega_{E_n}^c}|_{q}+
|E_n\chi_{\Omega_{E_n}}|_{p}\}$ are bounded.
\end{Lem}

Note that there is a continuous embedding
\begin{equation}\label{eqEmbedL6}
L^s(\R^3,\R^3)\subset L^{p,q}\hbox{ for } p\leq s\leq q.
\end{equation}
We show that the natural space for the energy functional $\E$ is
$$\D(\curl,p,q)$$ 
being the completion of $\mathcal{C}_0^{\infty}(\R^3,\R^3)$ with respect to the norm
$$\|E\|_{\curl,p,q}:=(|\curlop E|^2_2+|E|_{p,q}^2)^{1/2}.$$
The subspace of divergence-free vector fields is defined by
\[
\begin{aligned}
\U
 &= \left\{E\in \D(\curl,p,q)|\; \int_{\R^3}\langle E,\nabla \vp\rangle\,dx=0
        \text{ for any }\vp\in \cC^\infty_0(\R^3)\right\}\\
 &= \{E\in \D(\curl,p,q)|\; \div E=0\}
\end{aligned}
\]
where $\div E$ has to be understood in the distributional sense. Let $\D(\R^3,\R^3)$
be the completion of $\cC^{\infty}_0(\R^3,\R^3)$ with respect to the norm 
$$\|u\|_{\D}:=|\nabla u|_2$$
and let $\W$ be the completion of $\cC^{\infty}_0(\R^3)$ with respect to the norm 
$$\|w\|_{\W}:=|\nabla w|_{p,q}.$$
It is clear that $\W$ is linearly isometric to
$$\nabla\W:=\{\nabla w\in L^{p,q}:\; w\in\W\}.$$

The following Helmholtz's decomposition holds.
\begin{Lem}\label{defof} $\nabla\W$ is a closed subspace of $L^{p,q}$ and $\cl\U\cap\nabla\W=\{0\}$ in $L^{p,q}$. Moreover if $p\leq 6\leq q$, then
\begin{equation}\label{HelmholzDec}
\D(\curl,p,q)=\U\oplus\nabla \W
\end{equation}
and the norms $\|\cdot\|_{\D}$ and
$\|\cdot\|_{\curl,p,q}$ are equivalent on $\U$.
\end{Lem}
\begin{proof} Since $\W$ is a complete space, then clearly $\nabla\W$ is a closed subspace of $L^{p,q}$. Moreover $\cl\U\cap\nabla\W=\{0\}$ in $L^{p,q}$, hence $\U\cap\nabla\W=\{0\}$ in $\D(\curl,p,q)$. 
In view of the Helmholtz's decomposition any $\vp\in \cC^{\infty}_0(\R^3,\R^3)$ can be written as 
\begin{equation}\label{HelmSmooth}
\vp=\vp_1+\nabla \vp_2 
\end{equation}
such that $\vp_1\in \D(\R^3,\R^3)\cap \cC^{\infty}(\R^3,\R^3)$, $\div(\vp_1)=0$ and $\vp_2\in\cC^{\infty}(\R^3)$ is the Newton potential of $\div(\vp)$. Since $\vp$ has compact support, then $\nabla \vp_2\in L^{6}(\R^3,\R^3)\subset L^{p,q}$ and $\vp_1=\vp-\nabla\vp_2\in\U$. Observe that
$\curlop \curlop \vp_1 = -\Delta \vp_1$, hence 
$$|\curlop u|_2=|\nabla u|_2=\|u\|_\D$$
for any $u\in\U$. By the Sobolev embedding we have that $\U$ is
continuously embedded in $L^6(\R^3,\R^3)$ and by \eqref{eqEmbedL6} also in $L^{p,q}$. Therefore the norms $\|\cdot\|_{\D}$ and
$\|\cdot\|_{\curl,p,q}$ are equivalent on $\U$ and by the density argument we get the decomposition \eqref{HelmholzDec}.
\end{proof}


Let us assume that (F1), (F3) and (V) hold. We introduce a norm in $\U\times\W$ by the formula
$$\|(u,w)\|=(\|u\|_\D^2+\|w\|^2_{\W})^{\frac{1}{2}}$$
and consider a functional $\J:\U\times\W\to\R$ given by
\begin{equation}\label{eqJ}
\J(u,w):=\E(u+\nabla w)=\frac{1}{2}\int_{\R^3}|\nabla u|^2\,dx
+\frac{1}{2}\int_{\R^3}V(x)|u+\nabla w|^2\,dx-\int_{\R^3}F(x,u+\nabla w)\,dx
\end{equation}
for $(u,w)\in\U\times\W$.

The next Lemma \ref{LemEstimate} $(a)$ and \cite{BadialePisaniRolando}[Corollary 3.7] imply that $\E:\U\oplus\nabla\W\to\R$ and $\J:\U\times\W\to\R$ are well defined and of class $\cC^1$ with
\[
\begin{aligned}
&\E'(u+\nabla w)(\phi+\nabla\psi)
  = \J'(u,w)(\phi,\psi)\\
 &\hspace{1cm}
  = \int_{\R^3}\langle \curlop u,\curlop \phi\rangle\, dx
   + \int_{\R^3}V(x)\langle u+\nabla w,\phi+\nabla\psi\rangle\, dx
   - \int_{\R^3}\langle f(x,u+\nabla w),\phi+\nabla\psi\rangle \, dx
\end{aligned}
\]
for any $(u,w),(\phi,\psi)\in \U\times\W$. Thus we get the following observation.

\begin{Prop}\label{PropSolutE}
$(u,w)\in\U\times\W$ is a critical point of $\J$ if and only if
$E=u+\nabla w\in \U\oplus\nabla \W$ is a critical point of $\E$ in space $\U\oplus\nabla \W$ if and only if $E=u+\nabla w\in \U\oplus\nabla \W$ is a weak solution of \eqref{eq}, i.e.
$$\int_{\R^3}\langle E,\nabla\times\nabla\times\vp\rangle\,dx=\int_{\R^3}\langle -V(x)E+f(x,E),\vp\rangle\,dx\quad\hbox{ for any }\vp\in\cC^\infty_0(\R^3,\R^3),$$ and the electromagnetic energy \eqref{eq:EM_Energy} is finite for all $t$.
\end{Prop}
\begin{proof}
The first equivalence follows from Lemma \ref{HelmholzDec} and the above discussion. Let $E=u+\nabla w$ be a critical point of $\cE$ and $\vp\in\cC^{\infty}_0(\R^3,\R^3)$. We find a decomposition $\vp=\vp_1+\nabla\vp_2$ with $\vp_1\in\U$, $\vp_2\in\W$ according to \eqref{HelmSmooth} and observe that
\begin{eqnarray*}
\int_{\R^3}\langle E,\nabla\times\nabla\times\vp\rangle\,dx&=&
\int_{\R^3}\langle \nabla\times E,\nabla\times\vp\rangle\,dx=
\int_{\R^3}\langle \nabla\times u,\nabla\times\vp_1\rangle\,dx\\
&=&\int_{\R^3}\langle -V(x)E+f(x,E),\vp\rangle\,dx.
\end{eqnarray*}
Clearly if $E=u+\nabla w\in \U\oplus\nabla \W$ is a weak solution, then by the density argument we have $\cE'(E)=0$. Now
observe that
\begin{eqnarray*}
\cL(t)&=&\frac12\int_{\R^3}\cE\cD+\cB\cH\,dx\\
&=&\frac{1}{2\mu\omega^2}\int_{\R^3}
(-V(x)|E|^2+f(x,E)E)\cos^2(\omega t)+|\nabla\times E|^2\sin^2(\omega t)\,dx<\infty
\end{eqnarray*}
since $\cE'(E)(E)<\infty$.
\end{proof}

At the end of this section we collect some helpful inequalities.
\begin{Lem}\label{LemEstimate}
$\hbox{}$\\
$(a)$
If $E,F\in L^{p,q}$ then
\begin{eqnarray*}
\int_{\R^3}|V(x)||\langle E,F\rangle |\, dx
&\leq& 
(|V(x)E|_{\frac{p}{p-1}}+|V(x)E|_{\frac{q}{q-1}})|F|_{p,q},\\
&\leq& 
\Big((|V|_{\frac{p}{p-2}}^{\frac{p}{p-1}}|E\chi_{\Omega_E}|_p^{\frac{p}{p-1}}
+|V|_{\alpha}^{\frac{p}{p-1}}|E\chi_{\Omega_E^c}|_q^{\frac{p}{p-1}})^{\frac{p-1}{p}}\\
&&+(|V|_{\alpha}^{\frac{q}{q-1}}|E\chi_{\Omega_E}|_p^{\frac{q}{q-1}}
+|V|_{\frac{q}{q-2}}^{\frac{q}{q-1}}|E\chi_{\Omega_A^c}|_q^{\frac{q}{q-1}})^{\frac{q-1}{q}}\Big)|F|_{p,q}\\
&<&\infty,
\end{eqnarray*}
where $\frac{1}{\alpha}+\frac{1}{p}+\frac{1}{q}=1$.\\
$(b)$ If $E\in L^{p,q}$ then
\begin{eqnarray*}
\int_{\R^3} F(x,E)\,dx&\geq&c_1 \min\{|E|_{p,q}^p,|E|_{p,q}^q\}.
\end{eqnarray*}
$(c)$ If $E\in \D(\R^3,\R^3)$ then
\begin{eqnarray*}
\int_{\R^3}|\nabla E|^2+V(x)|E|^2\, dx &\geq&  c_3 |\nabla E|_2^2
\end{eqnarray*}
where $c_3:=1-|V|_{\frac{3}{2}} S^{-1}>0$.
\end{Lem}
\begin{proof}
$(a)$ Since $V\in L^{\frac{p}{p-2}}(\R^3)\cap L^{\frac{q}{q-2}}(\R^3)$ then for any
$\frac{q}{q-2}<\alpha<\frac{p}{p-2}$ we get
the following interpolation inequality
$$|V|_{\alpha}\leq |V|^\theta_{\frac{q}{q-2}}
|V|^{1-\theta}_{\frac{p}{p-2}}<+\infty$$
where $\theta\frac{q}{q-2}+(1-\theta)\frac{p}{p-2}=\alpha$.
Observe that by the H\"older inequality
\begin{eqnarray*}
\int_{\R^3}|V(x)E|^{\frac{p}{p-1}}\, dx&\leq& 
\int_{\Omega_E}|V(x)E|^{\frac{p}{p-1}}\, dx
+\int_{\Omega_E^c}|V(x)E|^{\frac{p}{p-1}}\, dx
\\
&\leq & |V|_{\frac{p}{p-2}}^{\frac{p}{p-1}}|E\chi_{\Omega_E}|_p^{\frac{p}{p-1}}
+|V|_{\alpha}^{\frac{p}{p-1}}|E\chi_{\Omega_E^c}|_q^{\frac{p}{p-1}}<\infty
\end{eqnarray*}
where $\frac{1}{\alpha}+\frac{1}{p}+\frac{1}{q}=1$.
Similarly we show that
$$\int_{\R^3}|V(x)E|^{\frac{q}{q-1}}\,dx
\leq |V|_{\alpha}^{\frac{q}{q-1}}|E\chi_{\Omega_E}|_p^{\frac{q}{q-1}}
+|V|_{\frac{q}{q-2}}^{\frac{q}{q-1}}|E\chi_{\Omega_E^c}|_q^{\frac{q}{q-1}}
<\infty.$$
Therefore for any $E\in L^{p,q}$
$$V(x)E\in L^{\frac{p}{p-1}}(\R^3,\R^3)\cap L^{\frac{q}{q-1}}(\R^3,\R^3)$$
and hence
\begin{eqnarray*}
\int_{\R^3}|V(x)||\langle E,F\rangle |\,dx&\leq&
(|V(x)E|_{\frac{p}{p-1}}+|V(x)E|_{\frac{q}{q-1}})|F|_{p,q}.
\end{eqnarray*}
$(b)$ Note that by (F3) and by Lemma \ref{LemBenFor} $(a)$
$$\int_{\R^3} F(x,E)\,dx \geq c_1\int_{\R^3} |E\chi_{\Omega_E^c}|^q +c_1\int_{\R^3} |E\chi_{\Omega_E}|^p
\geq c_1 \min\{|E|_{p,q}^p,|E|_{p,q}^q\}.$$
$(c)$ Let $E\in \D(\R^3,\R^3)$. Then it is enough to observe the following inequalities
$$-\int_{\R^3}V(x)|E|^2\, dx\leq \int_{\R^3}|V(x)||E|^2\, dx\leq 
|V
|_{\frac{3}{2}}|E|_6^2\leq 
|V|_{\frac{3}{2}}S^{-1}|\nabla E|_2^2.$$
\end{proof}

\section{Nehari-Pankov manifold}

From now on we assume that (F1)-(F5) and (V) hold. We introduce the Nehari-Pankov manifold for $\J$.
\begin{eqnarray}\label{DefOfN}
\mathcal{N} &:=& \{(u,w)\in\U\times\W|\; u\neq 0,\; 
\J'(u,w)(u,w)=0,\\\nonumber
&&\hbox{ and }\J'(u,w)(0,\psi)=0\,\hbox{ for any }\psi\in \W\}.
\end{eqnarray}
Observe that $E=u+\nabla w\in\cM$ if and only if $(u,w)\in\cN$. Moreover $\mathcal{N}$ contains all nontrivial critical points of $\J$. In general  $\cM$ and $\cN$ are not manifolds of $\cC^1$-class.

Let us define for any $u\in\U$
\begin{equation}
\A(u):=\{(t u,w)\in\U\times\W|\;t\geq 0\}
\end{equation}
and similarly as in \cite{BartschMederski}[Lemma 5.2] (cf. \cite{SzulkinWeth}[Proposition 2.3]) we get the following result.

\begin{Prop}\label{Propuv_N}
If $(u,w)\in\mathcal{N}$ then
$$\J(tu,tw+\psi)<\J(u,w)$$
for any $\psi\in\W$, $t\geq 0$ such that $(tu,tw+\psi)\neq (u,w)$. Thus $(u,w)\in\mathcal{N}$ is the unique global maximum of $\J|_{\A(u)}$.
\end{Prop}
\begin{proof}
Let $(u,w)\in\mathcal{N}$, $\psi\in\W$, $t\geq 0$ such that $(tu,tw+\psi)\neq (u,w)$. We take
$$D(t,\psi):=\J(t u,tw+\psi)-\J(u,w)$$
and observe that
\begin{eqnarray*}
D(t,\psi)&=&\frac{t^2-1}{2}\int_{\R^3}|\nabla u|^2+V(x)|u+\nabla w|^2\,dx\\
&&+\frac{1}{2}\int_{\R^3}V(x)(|tu+t\nabla w+\nabla \psi|^2-t^2|u+\nabla w|^2)\,dx\\
&&-\int_{\R^3}F(x,tu+t\nabla w+\nabla \psi)-F(x,u+\nabla w)\,dx\\
&=&\frac{t^2-1}{2}\int_{\R^3}|\nabla u|^2+V(x)|u+\nabla w|^2\,dx+t\int_{\R^3}V(x)\langle u+\nabla w,\nabla \psi\rangle\,dx\\
&&+\frac{1}{2}\int_{\R^3}V(x)|\nabla \psi|^2\,dx-\int_{\R^3}F(x,tu+t\nabla w+\nabla \psi)-F(x,u+\nabla w)\,dx.
\end{eqnarray*}
Since $(u,w)\in\mathcal{N}$, then
\begin{eqnarray*}
D(t,\psi)&=&\frac{1}{2}\int_{\R^3}V(x)|\nabla \psi|^2\,dx
+\int_{\R^3}\frac{t^2-1}{2}\langle f(x,u+\nabla w),u+\nabla w\rangle +F(x,u+\nabla w)\, dx\\
&&+\int_{\R^3} \langle tf(x,u+\nabla w),\nabla \psi\rangle-F(x,tu+t\nabla w+\nabla \psi)\, dx\\
&=&\frac{1}{2}\int_{\R^3}V(x)|\nabla \psi|^2\,dx
+\int_{\R^3}\langle f(x,u+\nabla w),\frac{t^2-1}{2}(u+\nabla w)+t\nabla \psi\rangle\, dx\\
&&+\int_{\R^3}F(x,u+\nabla w)-F(x,t(u+\nabla w)+\nabla \psi)\,dx.
\end{eqnarray*}
Define a map $\varphi:[0,+\infty)\times \R^3\to \R$ as follows
$$\varphi(t,x):=
\langle f(x,u+\nabla w),\frac{t^2-1}{2}(u+\nabla w)+t\nabla \psi\rangle+F(x,u+\nabla w)-F(x,t(u+\nabla w)+\nabla \psi).$$
Take $x\in\R^3$ such that $u(x)+\nabla w(x)\neq 0$.
Observe that by (F4) we have
$\varphi(0,x)<0$ and by (F3)
$$\lim_{t\to\infty}\varphi(t,x)=-\infty.$$
Let $t_0\geq 0$ be such that $\varphi(t_0,x)=\max_{t\geq 0}\varphi(t,x)$. If $t_0=0$ then $\vp(t,x)<0$ for any $t\geq 0$. Let us assume that $t_0>0$. Then $\partial_t \vp(t_0,x)=0$, i.e.
$$\langle f(x,u+\nabla w),t_0(u+\nabla w)+\nabla \psi\rangle=\langle f(x,t_0(u+\nabla w)+\nabla \psi),u+\nabla w\rangle$$
Note that if $\langle f(x,u+\nabla w),t_0(u+\nabla w)+\nabla \psi\rangle=0$ then by (F4)
\begin{eqnarray*}
\vp(t_0,x)
 &=& \langle f(x,u+\nabla w),\frac{-t_0^2-1}{2}(u+\nabla w)\rangle
      + F(x,u+\nabla w)-F(x,t_0(u+\nabla w)+\nabla\psi)\\
 &<& -t_0^2F(x,u+\nabla w)-F(x,t_0(u+\nabla w)+\nabla\psi)\\
 &\leq& 0.
\end{eqnarray*}
If $\langle f(x,u+\nabla w),t_0(u+\nabla w)+\nabla \psi\rangle\neq0$ then by (F5)
\begin{equation}\label{eq:phi}
\begin{aligned}
\vp(t_0,x)
 &= -\frac{(t_0-1)^2}{2}\langle f(x,u+\nabla w),u+\nabla w\rangle\\
 &\hspace{1cm}
     +t_0(\langle f(x,u+\nabla w),t_0(u+\nabla w)+\nabla\psi\rangle
     -\langle f(x,u+\nabla w),u+\nabla w\rangle)\\
&\hspace{1cm}+F(x,u+\nabla w)-F(x,t_0(u+\nabla w)+\nabla\psi)\\
&\leq-\frac{(\langle f(x,u+\nabla w),\nabla\psi\rangle)^2}{2\langle f(x,u+\nabla w),u+\nabla w\rangle}\\
 &\leq 0,
\end{aligned}
\end{equation}
and if $F(x,u+\nabla w)\neq F(x,t_0(u+\nabla w)+\nabla\psi)$ then $\vp(t_0,x)<0$. If $F(x,u+\nabla w)=F(x,t_0(u+\nabla w)+\nabla\psi)$ then (F5) yields
\[\langle f(x,u+\nabla w),t_0(u+\nabla w)+\nabla\psi\rangle
     \le \langle f(x,u+\nabla w),u+\nabla w\rangle.
\]
Therefore \eqref{eq:phi} implies
$$
\vp(t_0,x)\le-\frac{(t_0-1)^2}{2}\langle f(x,u+\nabla w),u+\nabla w\rangle.
$$
As a consequence, if $t_0\neq 1$ we deduce for $t\ge0$ that $\varphi(t,x)\leq\varphi(t_0,x)<0$. Now suppose $t_0=1$. If $\vp(t,x)=\vp(t_0,x)$ for some $0<t\neq t_0$ then $\partial_t\vp(t,x)=0$ and the above considerations imply $\vp(t,x)<0$. Summing up, we have shown that if $v(x)+\nabla w(x)\neq 0$ then $\vp(t,x)\leq 0$ for any $t\geq 0$ and $\vp(t,x)< 0$ if $t\neq 1$. Since $u+\nabla w\neq 0$ then we obtain 
$$D(t,\psi)<0$$
for any $t\neq 1$ and $\psi\in\W$.
Let us check the case $t=1$. Hence $\nabla \psi\neq 0$ and
$$D(1,\psi)<0$$
for $V<0$ a.e. on $\R^3$. If $V=0$ a.e. on a subset of positive measure then
by (F2)
$$\varphi(1,x)=
f(x,u+\nabla w)(\nabla \psi)+F(x,u+\nabla w)-F(u+\nabla w+\nabla \psi)<0$$
provided that $\nabla \psi(x)\neq 0$. Finally we get
$$D(t,\psi)=\J(t u,tw+\psi)-\J(u,w)<0$$
if $(tu,tw+\psi)\neq (u,w)$. 
\end{proof}

Let us consider $I:L^{p,q}\to \R$ defined by formula \eqref{DefOfI}. Moreover $\cI:L^{p,q}\times \W\to\R$ is given by
\begin{equation}\label{DefOfXi}
\cI(u,w):=I(u+\nabla w)\hbox{ for }(u,w)\in L^{p,q}\times\W.
\end{equation}
Similarly as above by Lemma \ref{LemEstimate} $(a)$ and \cite{BadialePisaniRolando}[Corollary 3.7] we check that $I$, $\cI$ are of $\cC^1$-class. In view of (F2) we have that $I$, $\cI$ are strictly convex. Moreover the following property holds.

\begin{Lem}\label{LemConvWeakIpliesStrong}
If $E_n\rightharpoonup E$ in $L^{p,q}$ and
$I(E_n)\to I(E)$
then $E_n\to E$ in $L^{p,q}$.
\end{Lem}

Before we prove the above lemma we need a variant of Brezis-Lieb result for sequences in $L^{p,q}$ (cf. \cite{BrezisLieb}). 
\begin{Lem}\label{LemBrezLieb}
Let $\{E_n\}$ be a bounded sequence in $L^{p,q}$ such that
$E_n\to E$ a.e. on $\R^3$.
Then
$$\lim_{n\to+\infty}\int_{\R^3}F(x,E_n)-F(x,E_n-E)\, dx=\int_{\R^3}F(x,E)\, dx.$$
\end{Lem}
\begin{proof}
Note that
\begin{eqnarray*}
\int_{\R^3}F(x,E_n)-F(x,E_n-E)\, dx
&=&\int_{\R^3}\int_0^1\frac{d}{dt}F(x,E_n-E+tE)\, dtdx\\
&=&\int_0^1\int_{\R^3}\langle f(x,E_n-E+tE),E\rangle\, dxdt
\end{eqnarray*}
and $f(x,E_n-E+tE)$ is bounded
in $L^{\frac{p}{p-1}}(\R^3,\R^3)\cap L^{\frac{q}{q-1}}(\R^3,\R^3)$.
Thus for any $\Omega\subset\R^3$
\begin{eqnarray*}
\int_{\Omega}|\langle f(x,E_n-E+tE),E\rangle|\, dx
\leq (|f(x,E_n-E+tE)|_{\frac{p}{p-1}}
+|f(x,E_n-E+tE)|_{\frac{q}{q-1}})
|E\chi_{\Omega}|_{p,q}.
\end{eqnarray*}
In view of Lemma \ref{LemBenFor} $(a)$,
for any $\eps>0$ there is $n_0\in\N$ and $\delta>0$
such that for any $\Omega$ with $|\Omega|<\delta$ the following inequality
holds 
$$\int_{\Omega}|\langle f(x,E_n-E+tE),E\rangle|\, dx < \eps$$
for any $n\geq n_0$.
Thus $(\langle f(x,E_n-E+tE),E\rangle )_n$ is uniformly integrable.
Moreover for any $\eps>0$ there is $n_0\in\N$ and $\Omega\subset\R^3$ with 
$|\Omega|<+\infty$
such that for for any $n\geq n_0$
$$\int_{\Omega^c}\langle f(x,E_n-E+tE),E\rangle\, dx < \eps.$$
Hence $(\langle f(x,E_n-E+tE),E\rangle )_n$ is tight.
Since $E_n(x)-E(x)\to 0$ a.e. on $\R^3$ then
 in view of the Vitali convergence theorem $\langle f(x,tE)E\rangle$ is integrable and
$$\int_{\R^3}F(x,E_n)-F(x,E_n-E)\, dx\to
\int_0^1\int_{\R^3}\langle f(x,tE)E\rangle\, dxdt=\int_{\R^3}F(x,E)\, dx.$$
\end{proof}

\begin{altproof}{Lemma \ref{LemConvWeakIpliesStrong}}
We show that (up to a subsequence) $E_n(x)\to E(x)$ a.e. on $\R^3$.
Since $I(E_n)\to I(E)$ by lower semicontinuity we have
\begin{eqnarray}\label{LemConvEq1}
\lim_{n\to\infty}\int_{\R^3}-\frac{1}{2}V(x)|E_n|^2\,dx&=&
\int_{\R^3}-\frac{1}{2}V(x)|E|^2\,dx,\\
\lim_{n\to\infty}\int_{\R^3}F(x,E_n)\,dx&=&
\int_{\R^3}F(x,E)\,dx.\label{LemConvEq2} 
\end{eqnarray}
If $V<0$ a.e. on $\R^3$ then passing to a subsequence $(-V(x))^{1/2}E_n\weakto (-V(x))^{1/2}E$ in $L^2(\R^3,\R^3)$ and by \eqref{LemConvEq1} we get
$(-V(x))^{1/2}E_n\to (-V(x))^{1/2}E$ in $L^2(\R^3,\R^3)$. Thus
$E_n\to E$ a.e. on $\R^3$. Assume that $F$ is uniformly strictly convex in $u\in\R^3$ (see (F2)). Then for any $0<r\leq R$
$$m:=\inf_{\genfrac{}{}{0pt}{}{x\in\R^3,u_1,u_2\in\R^3}{r\leq|u_1-u_2|,|u_1|,|u_2|\leq R} }\;
\frac{1}{2}(F(x,u_1)+F(x,u_2))-F\Big(x,\frac{u_1+u_2}{2}\Big)>0$$
Observe that by the convexity of $F$ in $u\in\R^3$
$$
0
 \leq \limsup_{n\to\infty}\int_{\R^3}\frac12(F(x,E_n) + F(x,E))
        - F\left(x,\frac{E_n+E}{2}\right)\,dx
 \leq 0.
$$
Therefore setting
$$
\Om_n:=\{x\in\Om|\;|E_n-E|\geq r,\;|E_n|\leq R,\;
          |E|\leq R\}
$$
there holds
$$
\mu(\Om_n)m
 \leq \int_{\R^3}\frac12(F(x,E_n) + F(x,E))
       - F\left(x,\frac{E_n+E}{2}\right)\, dx
$$
and thus $\mu(\Om_n)\to 0$ as $n\to\infty$. Since $0<r\leq R$ are arbitrary chosen, we deduce
$$
E_n\to E\hbox{ a.e.\ on }\R^3.
$$
In view of Lemma \ref{LemBrezLieb}  we obtain
$$\int_{\R^3}F(x,E_n)\, dx-\int_{\R^3}F(x,E_n-E)\, dx\to\int_{\R^3}F(x,E)\, dx$$
and thus
$$\int_{\R^3}F(x,E_n-E)\, dx\to0.$$
By Lemma \ref{LemEstimate} $(b)$
we get $|E_n-E|_{p,q}\to 0$.
\end{altproof}

Now we are able to apply the critical point theory on the Nehari-Pankov manifold deve\-loped in \cite{BartschMederski}[Section 4]. Namely we get the following result.

\begin{Prop}\label{PropDefOfm(u)}$\hbox{}$\\
$(a)$ For any $u\in\U\setminus\{0\}$, there are unique $t=t(u)>0$ and $w\in\W$ such that
$$m(u):=(tu,w)\in\mathcal{N}\cap\A(u)$$ and
$$\J(m(u))=\sup_{\A(u)}\J.$$
Moreover $m:\U\setminus\{0\}\to \mathcal{N}$ is continuous and $m|_{S_\U}$ is a homeomorphism, where 
$$S_{\U}:=\{u\in\U|\;\|u\|_{\D}=1\}.$$
$(b)$ There is a sequence $(u_n)\subset S_\U$ such that $(m(u_n))$ is a $(PS)_{c}$-sequence for $\J$ at level $c$, i.e. $\J(m(u_n))\to c$ and $\J'(m(u_n))\to 0$ as $n\to\infty$, where
$$c:=\inf_{(u,w)\in\mathcal{N}}\J(u,w)>0.$$
\end{Prop}
\begin{proof}
Setting $X:=\U\times\W$,
$X^+:=\U\times\{0\}$ and $\tX:=\{0\}\times\V$ we check assumptions (A1)-(A4), (B1)-(B3) of \cite{BartschMederski}[Theorem 4.1, Proposition 4.2] for $\J:X\to\R$ 
of the form:
$$\J(u,w)=\frac{1}{2}\|u\|^2_\D-\cI(u,w),$$
The convexity of $\cI\in \cC^1(L^{p,q},\R)$, (V), (F3) and Lemma \ref{LemConvWeakIpliesStrong} yield:
\begin{itemize}
 \item[(A1)] $\cI|_{\U\times\W}\in \cC^1(\U\times\W,\R)$ and $\cI(u,w)\geq I(0,0)=0$ for any $(u,w)\in \U\times\W$.
 \item[(A2)] If $u_n\to u$ in $\U$, $w_n\weakto w$ in $\W$,  then $\liminf_{n\to\infty}\cI(u_n,w_n)\geq \cI(u,w)$.
 \item[(A3)] If $u_n\to u$ in $\U$, $w_n\weakto w$ in $\W$ and $\cI(u_n,w_n)\to \cI(u,w)$, then $(u_n,w_n)\to (u,w)$.
\end{itemize}
Moreover the following condition holds.
\begin{itemize}
 \item[(A4)] There exists $r>0$ such that $\inf_{\|u\|_\D=r}\J(u,0)>0$.
\end{itemize}
Indeed, in view of Lemma \ref{LemEstimate} $(c)$ and by (F3) for any $u\in\U$
$$\J(u,0)\geq c_3\|u\|_\D^2 -\int_{\R^3}F(x,u)\,dx
\geq c_3\|u\|_\D^2 -\frac{c_2}{2}\int_{\R^3}|u|^6\,dx
\geq c_3\|u\|_\D^2 -\frac{c_2}{2}S^{-3}\|u\|_\D^6$$
and thus (A4) is satisfied.
Moreover by Lemma \ref{LemEstimate} $(b)$ it is easy to verify
\begin{itemize}
 \item[(B1)] $\|u\|_\D+\cI(u,w)\to\infty$ as $\|(u,w)\|\to\infty$.
\end{itemize}
We prove the following condition.
\begin{itemize}
 \item[(B2)] $\cI(t_n(u_n,w_n))/t_n^2\to\infty$ if $t_n\to\infty$ and $u_n\to u$ for some $u\neq 0$ as $n\to\infty$.
\end{itemize}
Observe that by Lemma \ref{LemEstimate} $(b)$
\begin{eqnarray*}
\cI(t_n(u_n,w_n))&\geq&
\int_{\R^3}F(x,t_n(u_n+\nabla w_n))\,dx\\
&\geq& c_1
\min\{ |t_n u_n+t_n\nabla w_n|_{p,q}^p,|t_n u_n+t_n\nabla w_n|_{p,q}^q\}\\
&\geq&c_1t_n^2
\min\{ t_n^{p-2}|u_n+\nabla w_n|_{p,q}^p,t_n^{q-2}|u_n+\nabla w_n|_{p,q}^q\}.
\end{eqnarray*} 
If $\liminf_{n\to\infty}|u_n+\nabla w_n|_{p,q}=0$ as $n\to\infty$, then passing to a subsequence we get 
$$|u+\nabla w_n|_{p,q}\to 0.$$
Hence  we get a contradiction to the assumption $u\neq 0$. Therefore $|u_n+\nabla w_n|_{p,q}$ is bounded away from $0$ and $\cI(t_n(u_n,w_n))/t_n^2\to\infty$ as $n\to\infty$.
Finally the arguments provided in proof of Proposition \ref{Propuv_N} show that:
\begin{itemize}
 \item[(B3)]
$\frac{t^2-1}{2}\langle \cI'(u,w),(u,w)\rangle+t\langle \cI'(u,w),(0,\psi)+\cI(u,w)-\cI(tu,tw+\psi)<0$
for any $t\geq 0$, $u\in\U$ and $w,\psi\in\W$ such that $(tu,tw+\psi)\neq (u,w)$.
\end{itemize}
Finally we obtain statements $(a)$ and $(b)$ applying \cite{BartschMederski}[Theorem 4.1 a), Proposition 4.2]. The continuity of $m:\U\setminus\{0\}\to\N$ follows directly from arguments given in proof of \cite{BartschMederski}[Theorem 4.1].  
\end{proof}

Since there is no compact embedding of $\U$ into $L^{p,q}$,
the critical point theory provided in \cite{BartschMederski}[Section 4] is not sufficient to show that $c=\inf_{\cN}\J$ is achieved by a critical point of $\J$. Therefore in the next Section \ref{SectAnalysis} we provide an analysis of bounded sequences in $\D(\R^3,\R^3)$ and of bounded sequences of the Nehari-Pankov manifold.

\section{Analysis of bounded sequences}\label{SectAnalysis}

We need further properties of $\cI$.
\begin{Lem}\label{LemDefofW}$\hbox{}$\\
$(a)$ There is the unique continuous map $w:L^{p,q}\to\W$ such that
\begin{equation}\label{DefofW(u)}
\cI(u,w(u))=\inf_{w\in\W}\cI(u,w).
\end{equation}
$(b)$ $w$ maps bounded sets into bounded sets and $w(0)=0$.\\
$(c)$ If $u\in\U\setminus\{0\}$ then 
$m(u)=(t(u)u,w(t(u)u))$.
\end{Lem}
\begin{proof}
$(a)$ Let $u\in L^{p,q}$. 
Since  $\W\ni w\mapsto \cI(u,w)\in\R$ is continuous, strictly convex and coercive, then there exists a unique $w(u)\in\W$ such that (\ref{DefofW(u)}) holds.
We show that the map $w:L^{p,q}\to\W$ is continuous. Let $u_n\to u$ in $L^{p,q}$. Since 
\begin{equation}\label{Lem_m_map_ineq}
0\leq \cI(u_n,w(u_n))\leq \cI(u_n,0)
\end{equation}
we obtain that $w(u_n)$ is bounded and we may assume that $w(u_n)\rightharpoonup w_0$ for some $w_0\in\W$.
Observe that by the (sequentially) lower semi-continuity of $\cI$ we get
$$\cI(u,w(u))\leq \cI(u,w_0)\leq \liminf_{n\to\infty} 
\cI(u_n,w(u_n))\leq 
\liminf_{n\to\infty} 
\cI(u_n,w(u))=\cI(u,w(u)).$$
Hence $w(u)=w_0$ and by Lemma \ref{LemConvWeakIpliesStrong} we have $u_n+\nabla w(u_n)\to u+\nabla w(u)$ in $L^{p,q}$. Thus $w(u_n)\to w(u)$ in $\W$.\\ 
$(b)$ This follows from inequality (\ref{Lem_m_map_ineq}) and Lemma \ref{LemEstimate} $(b)$.\\
$(c)$ Let $u\in\U\setminus\{0\}$ and $m(u)=(t(u)u,w)$. Note that
$$\J(m(u))=\frac{1}{2}\|t(u)u\|^2_{\D}+\cI(t(u)u,w)\leq 
\frac{1}{2}\|t(u)u\|^2_{\D}-\cI(t(u)u,w(t(u)u))=\J(t(u)u,w(t(u)u)).$$
In view of Proposition \ref{PropDefOfm(u)} $(a)$ we get 
$m(u)=(t(u)u,w(t(u)u)$.
\end{proof}

Below we analyse a bounded sequence $(u_n)$ in $\D(\R^3,\R^3)$ and provide a possibly infinite splitting of $\lim_{n\to\infty}\cI(u_n,w(u_n))$.

\begin{Lem}\label{LemSplit1}
If $(u_n)$ is bounded in $\D(\R^3,\R^3)$ then, up to a subsequence,
there is $N\in \N\cup \{\infty\}$ and there are sequences $(\bar{u}_i)_{i\in\N}\subset \D(\R^3,\R^3)$, $(x_n^i)_{n\geq i}\subset \Z^3$ such that $x_n^0=0$ and the following conditions hold:\\
$(a)$ If $N<\infty$ then $\bar{u}_i\neq 0$ for $1\leq i\leq N$ and $\bar{u}_i=0$ for $i>N$, if $N=\infty$ then $\bar{u}_i\neq 0$ for all $i\geq 1$,\\
$(b)$ $u_n(\cdot+x_n^i)\rightharpoonup \bar{u}_i\hbox{ in }\D(\R^3,\R^3)$ for any $0\leq i < N+1$ \textnormal{(\footnote{If $N=\infty$ then $N+1=\infty$ as well.})},\\
$(c)$ $u_n(\cdot+x_n^i)\to\bar{u}_i\hbox{ in }L_{loc}^{p,q}$ and a.e. in $\R^3$ for any $0\leq i < N+1$,\\
$(d)$ $u_n-\sum_{i=0}^n\bar{u}_i(\cdot-x_n^i)\to 0\hbox{ in }L^{p,q}$.\\
Moreover\\
$(e)$ $\nabla w(u_n)\rightharpoonup \nabla w(\bar{u}_0)$ and $\nabla w_0(u_n)(\cdot+x_n^i)\rightharpoonup \nabla w_0(\bar{u}_i)\hbox{ in }L^{p,q}$ for any $1\leq i < N+1$,\\
$(f)$ $\nabla w(u_n)\to \nabla w(\bar{u}_0)$ and $\nabla w_0(u_n)(\cdot+x_n^i)\to \nabla w_0(\bar{u}_i)\hbox{ in }L_{loc}^{p,q}$ and a.e. in $\R^3$ for any $1\leq i < N+1$,\\
$(g)$ $\nabla w(u_n)-\nabla w(\bar{u}_0)-\sum_{i=1}^n\nabla w_0(\bar{u}_i)(\cdot-x_n^i)\to 0\hbox{ in }L^{p,q}$,\\
$(h)$ $\lim_{n\to\infty}\cI(u_n,w(u_n))=\cI(\bar{u}_0,w(\bar{u}_0))+\sum_{i=1}^{\infty} \cI_0(\bar{u}_i,w_0(\bar{u}_i))<\infty$,\\
where $w_0$ and $\cI_0$ are maps given by $(\ref{DefofW(u)})$ and $(\ref{DefOfXi})$ under assumption $V=0$.
\end{Lem}
\begin{proof}
We may assume that $u_n\rightharpoonup \bar{u}_0$ in $\D(\R^3,\R^3)$ for some $\bar{u}_0\in\D(\R^3,\R^3)$. Then $(b)$, $(c)$ and $(d)$ has been obtained in proof of \cite{DAprileSiciliano}[Lem. 4.2]. Indeed, recall that using a variant of the concentration compactness argument \cite{DAprileSiciliano}[Lem. 4.1] we show that there is $N\in \N\cup \{\infty\}$ and there are sequences $(\bar{u}_i)_{i\in\N}\subset \D(\R^3,\R^3)$, $(x_n^i)_{n\geq i}\subset \R^3$ and positive numbers $(c_i)_{i\in\N}$ such that $x_n^0=0$ and, up to a subsequence, $(b)$, $(d)$ hold. Moreover for any $0\leq i <N+1$, $n\geq i$
\begin{eqnarray}
\label{Eqstronglok_u_n}
&&u_n(\cdot+x_n^i)\chi_{B(0,n)}\to\bar{u}_i\hbox{ in }L^{p,q},\\
\label{Eqxnxm}
&&|x_n^i-x_n^j|\geq n-2r\hbox{ for } j\neq i, 0\leq j< N+1\\\label{EqIntegralunSum}
&&\int_{B(x_n^{i+1},r)}\Big|u_n-\sum_{j=0}^{i}\bar{u}_j(\cdot -x_n^j)\Big|^2\, dx \geq c_{i+1},
\end{eqnarray}
where $r>0$.
If $N<\infty$ then we take $\bar{u}_i=0$ for $i> N$. If $N=\infty$ then the above conditions hold for any $i\geq 0$. Observe that we may assume that 
$(x_n^i)_{n\geq i}\subset \Z^3$ for $r>\sqrt{3}$.
Hence the local convergence in $(c)$ follows directly from (\ref{Eqstronglok_u_n}). 
Moreover the boundedness of $(w(u_n))_{n\in\N}$ and $(w_0(u_n))_{n\in\N}$ in $\W$ implies that we may assume
\begin{eqnarray}\label{Eqweakw}
&&\nabla w(u_n)\rightharpoonup \nabla \bar{w}_0\hbox{ in }L^{p,q},\\\label{Eqweakw2}
&&\nabla w_0(u_n)(\cdot+x_n^i)\rightharpoonup \nabla \bar{w}_i\hbox{ in }L^{p,q}\hbox{ for } i\geq 1.
\end{eqnarray}
Observe that
$(a)$, $(e)$ -- $(h)$ are a consequence  of the following claims and the almost everywhere convergence in $(c)$ and $(f)$ follows from the local convergence in $L^{p,q}$ (see \cite{BadialePisaniRolando}[Prop. 2.8]).\\
{\it Claim 1.} $\bar{u}_{i}\neq 0$ for $1\leq i<N+1$.\\
Let $0\leq i<N$. Observe that (\ref{EqIntegralunSum}) implies that
\begin{eqnarray*}
0&<&\sqrt{c_{i+1}}\leq\Big(\int_{B(x_n^{i+1},r)}\Big|u_n-\sum_{j=0}^{i}\bar{u}_j(\cdot -x_n^j)\Big|^2\, dx\Big)^{\frac{1}{2}}\\
&\leq& \Big(\int_{B(0,r)}|u_n(\cdot +x_n^{i+1})|^2\, dx\Big)^{\frac{1}{2}}+
\sum_{j=0}^{i} \Big(\int_{B(x_n^{i+1}-x_n^j,r)}|\bar{u}_j|^2\, dx\Big)^{\frac{1}{2}}.
\end{eqnarray*}
From (\ref{Eqstronglok_u_n}) we easily see that $\bar{u}_j\chi_{B(x_n^{i+1}-x_n^j,r)}\to 0$ in $L^{p,q}$ and then $\bar{u}_j(\cdot +x_n^{i+1}-x_n^j)\chi_{B(0,r)}\to 0$ in $L^{p,q}$ for any $0\leq j \leq i$. In view of \cite{BadialePisaniRolando}[Prop. 2.14] we know that $\bar{u}_j(\cdot +x_n^{i+1}-x_n^j)\to 0$ in $L^2(B(0,r),\R^3)$.
Therefore, up to a subsequence, $u_n(\cdot +x_n^{i+1})\to \bar{u}_{i+1}$ in $L^2(B(0,r),\R^3)$ and then
$$0<\sqrt{c_{i+1}}\leq \Big(\int_{B(0,r)} |\bar{u}_{i+1}|^2\, dx\Big)^{\frac{1}{2}}.$$
Thus $\bar{u}_{i+1}\neq 0$ for $0\leq i<N$. 
\\
{\it Claim 2.} There holds
\begin{equation}\label{EqSeries}
\sum_{i=1}^{\infty}
\cI_0(\bar{u}_i,w_0(\bar{u}_i))<+\infty.
\end{equation}
Indeed, observe that Lemma \ref{LemEstimate} $(b)$, the weak lower semicontinuity of $\cI_0$ and conditions $(b)$, (\ref{Eqweakw2}) imply that
\begin{eqnarray*}
\sum_{i=1}^k \cI_0(\bar{u}_i,w_0(\bar{u}_i))&\leq&
 \sum_{i=1}^k \cI_0(\bar{u}_i,w_i)\\
&\leq& \sum_{i=1}^k \liminf_{n\to\infty}\cI_0(u_n(\cdot+x_n^i)\chi_{B(0,\frac{n-2}{2})},w_0(u_n)(\cdot+x_n^i)\chi_{B(0,\frac{n-2}{2})})\\
&\leq& \liminf_{n\to\infty}\sum_{i=1}^k \cI_0(u_n\chi_{B(x_n^i,\frac{n-2}{2})},w_0(u_n)\chi_{B(x_n^i,\frac{n-2r}{2})})\leq
 \liminf_{n\to\infty} \cI_0(u_n,w_0(u_n))
\end{eqnarray*}
for any $k\in\N$. By Lemma \ref{LemDefofW} $(b)$ we obtain that $(\cI_0(u_n,w_0(u_n)))_{n\in\N}$ is bounded. Therefore (\ref{EqSeries}) holds.\\
{\it Claim 3.} Up to a subsequence
\begin{equation}\label{LemClaim2}
\lim_{n\to\infty}\int_{\bigcup_{j=1}^nB(x_n^j,\frac{n-2r}{2})}V(x)\Big|\sum_{i=1}^n(\bar{u}_i+\nabla w_0(\bar{u}_i))(\cdot-x_n^i))\Big|^2\,dx=0.
\end{equation}
We show that up to a subsequence $\sum_{i=1}^n(\bar{u}_i+\nabla w_0(\bar{u}_i))(\cdot-x_n^i))\chi_{\bigcup_{j=1}^nB(x_n^j,\frac{n-2r}{2})}$ is bounded in $L^{p,q}$. Let us fix $k\geq 1$ and observe that
\begin{eqnarray*}
&&\cI_0\Big(\sum_{i=1}^k\bar{u}_i(\cdot-x_n^i)\chi_{\bigcup_{j=1}^kB(x_n^j,\frac{n-2r}{2})},\sum_{i=1}^k w_0(\bar{u}_i)(\cdot-x_n^i))\chi_{\bigcup_{j=1}^kB(x_n^j,\frac{n-2r}{2})}\Big)\\
&&= \sum_{j=1}^k \cI_0\Big(\sum_{i=1}^k\bar{u}_i(\cdot-x_n^i)\chi_{B(x_n^j,\frac{n-2r}{2})},\sum_{i=1}^k w_0(\bar{u}_i)(\cdot-x_n^i))\chi_{B(x_n^j,\frac{n-2r}{2})}\Big).
\end{eqnarray*}
Let 
$$v_0:=\nabla w(\bar{u}_0)$$
and for $i\geq 1$
$$v_i:=\nabla w_0(\bar{u}_i).$$
Since 
$$B\Big(x_n^j-x_n^i,\frac{n-2r}{2}\Big)\subset \R^3\setminus B\Big(0,\frac{n-2r}{2}\Big)$$
for $i\neq j$,
then for given $0\leq j\leq k$
$$\Big|\sum_{0\leq i\leq k,i\neq j}(\bar{u}_i+v_i)(\cdot-x_n^i)\chi_{B(x_n^j,\frac{n-2r}{2})}\Big|_{p,q}\leq
\sum_{0\leq i\leq k,i\neq j}\Big|(\bar{u}_i+v_i)\chi_{B(x_n^j-x_n^i,\frac{n-2r}{2})}\Big|_{p,q}\to 0$$
as $n\to\infty$. Then for any $k\geq 1$ there is sufficiently large $n=n(k)$ such that
\begin{eqnarray}\label{eq:Lem_estimation_important}
&&\cI_0\Big(\sum_{i=1}^k\bar{u}_i(\cdot-x_n^i)\chi_{\bigcup_{j=1}^kB(x_n^j,\frac{n-2r}{2})},\sum_{i=1}^kw_0(\bar{u}_i)(\cdot-x_n^i)\chi_{\bigcup_{j=1}^kB(x_n^j,\frac{n-2r}{2})}\Big)\\\nonumber
&&\leq \sum_{j=1}^k \cI_0(\bar{u}_j\chi_{B(0,\frac{n-2r}{2})},w_0(\bar{u}_j)\chi_{B(0,\frac{n-2r}{2})}\Big)+\frac{1}{k}\\
&&\leq \sum_{j=1}^\infty \cI_0(\bar{u}_j, w_0(\bar{u}_j))+\frac{1}{k} \nonumber
\end{eqnarray}
and by passing to a subsequence, in view of {\em Claim 2} and Lemma \ref{LemEstimate} (b)  we get 
the boundedness of  $\sum_{i=1}^n(\bar{u}_i+\nabla w_0(\bar{u}_i))(\cdot-x_n^i))\chi_{\bigcup_{j=1}^nB(x_n^j,\frac{n-2r}{2})}$ in $L^{p,q}$.
Now, note that similarly as in Lemma \ref{LemEstimate} $(a)$ we obtain
\begin{eqnarray*}
&&\int_{\bigcup_{j=1}^nB(x_n^j,\frac{n-2r}{2})}V(x)\Big|\sum_{i=1}^n(\bar{u}_i+\nabla w_0(\bar{u}_i))(\cdot-x_n^i))\Big|^2\,dx\\
&&\leq C\max\{|V\chi_{\bigcup_{j=1}^nB(x_n^j,\frac{n-2r}{2})}|_{\frac{p}{p-2}},|V\chi_{\bigcup_{j=1}^nB(x_n^j,\frac{n-2r}{2})}|_{\frac{q}{q-2}}\}
\end{eqnarray*}
for some constant $C>0$. Since 
$$\bigcup_{j=1}^nB\Big(x_n^j,\frac{n-2r}{2}\Big)\subset \R^3\setminus B\Big(0,\frac{n-2r}{2}\Big)$$
then we get (\ref{LemClaim2}).\\
{\it Claim 4.}  Up to a subsequence
\begin{eqnarray}\label{Ineq15}
&&\limsup_{n\to\infty}\cI\Big(\sum_{i=0}^n\bar{u}_i(\cdot-x_n^i),w(\bar{u}_0)+
\sum_{i=0}^nw_0(\bar{u}_i)(\cdot-x_n^i)\Big)\\\nonumber
&&\leq 
\cI(\bar{u}_0,w(\bar{u}_0))+\sum_{i=1}^{\infty} \cI_0(\bar{u}_i,w_0(\bar{u}_i)).
\end{eqnarray}
Similarly as above
$$v_0:=\nabla w(\bar{u}_0)\hbox{ and } v_i:=\nabla w_0(\bar{u}_i)\hbox{ for }i\geq 1.$$
Note that $B(x_n^i,\frac{n-2r}{2})\cap B(x_n^j,\frac{n-2r}{2})=\emptyset$  for $i\neq j$ and
\begin{eqnarray}\label{Eqnxi}
&&\cI\Big(\sum_{i=0}^n\bar{u}_i(\cdot-x_n^i),w(\bar{u}_0)+
\sum_{i=0}^nw_0(\bar{u}_i)(\cdot-x_n^i)\Big)=\\\nonumber
&&\int_{B(0,\frac{n-2r}{2})}-\frac{1}{2}V(x)\Big|\sum_{i=0}^n(\bar{u}_i+v_i)(\cdot-x_n^i))\Big|^2+F\Big(x,\sum_{i=0}^n(\bar{u}_i+v_i))(\cdot-x_n^i)\Big)\,dx\\\nonumber
&&+\int_{\bigcup_{j=1}^nB(x_n^j,\frac{n-2r}{2})}-\frac{1}{2}V(x)\Big|\sum_{i=0}^n(\bar{u}_i+v_i)(\cdot-x_n^i))\Big|^2\,dx
\\\nonumber
&&+\sum_{j=1}^n\int_{B(x_n^j,\frac{n-2r}{2})}F\Big(x,\sum_{i=0}^n(\bar{u}_i+v_i)(\cdot-x_n^i)\Big)\,dx\\\nonumber
&&+\int_{\R^3\setminus \bigcup_{j=0}^n B(x_n^j,\frac{n-2r}{2})}-\frac{1}{2}V(x)\Big|\sum_{i=0}^n(\bar{u}_i+v_i)(\cdot-x_n^i))\Big|^2+F\Big(x,\sum_{i=0}^n(\bar{u}_i+v_i)(\cdot-x_n^i)\Big)\,dx.
\end{eqnarray}
Moreover observe that for given $k\geq 0$
$$ \Big|\sum_{i=0}^k
(\bar{u}_i+v_i)(\cdot-x_n^i)\chi_{\R^3\setminus \bigcup_{j=0}^n B(x_n^j,\frac{n-2r}{2})}\Big|_{p,q}\leq 
\sum_{i=0}^k\Big|\
(\bar{u}_i+v_i)\chi_{\R^3\setminus B(0,\frac{n-2r}{2})}\Big|_{p,q}\to 0
$$
as $n\to\infty$. Therefore, up to a subsequence, 
$$\sum_{i=0}^n
(\bar{u}_i+v_i)(\cdot-x_n^i)\chi_{\R^3\setminus \bigcup_{j=0}^n B(x_n^j,\frac{n-2r}{2})}\to 0 \hbox{ in } L^{p,q}$$
and from (\ref{LemClaim2}), \eqref{eq:Lem_estimation_important},  (\ref{Eqnxi}) we obtain (\ref{Ineq15}).\\
{\it Claim 5.} Up to a subsequence we have
\begin{eqnarray}\label{Claim4_1}
&&\nabla w(u_n)\chi_{B(0,\frac{n-2r}{2})}\to \nabla w(\bar{u_0})\hbox{ in }L^{p,q},\\\label{Claim4_2}
&&\nabla w_0(u_n)(\cdot+x_n^i)\chi_{B(0,\frac{n-2r}{2})}\to \nabla w_0(\bar{u_i})\hbox{ in }L^{p,q},
\end{eqnarray}
as $n\to\infty$. Moreover $(h)$ holds.\\
Let $\eps>0$, $k\in\N$. Then for sufficiently large $n$
\begin{eqnarray*}
&&\hspace{-7mm}\eps+
\cI\Big(\sum_{i=0}^n\bar{u}_i(\cdot-x_n^i),
w(\bar{u}_0)+\sum_{i=1}^nw_0(\bar{u}_i)(\cdot-x_n^i)\Big)\\
&\geq&
\cI\Big(u_n,w(\bar{u}_0)+\sum_{i=1}^nw_0(\bar{u}_i)(\cdot-x_n^i)\Big)
\geq 
\cI(u_n,w(u_n))\\
&\geq&\sum_{i=0}^k \cI(u_n\chi_{B(x_n^i,\frac{n-2r}{2})},w(u_n)\chi_{B(x_n^i,\frac{n-2r}{2})})\\
&\geq& -\eps+\sum_{i=0}^k \liminf_{n\to\infty}\cI(u_n\chi_{B(x_n^i,\frac{n-2r}{2})},w(u_n)\chi_{B(x_n^i,\frac{n-2r}{2})})\\
&\geq& -\eps+\liminf_{n\to\infty}\cI(u_n\chi_{B(0,\frac{n-2r}{2})},w(u_n)\chi_{B(0,\frac{n-2r}{2})})\\
&&+\sum_{i=1}^k \liminf_{n\to\infty}\cI_0(u_n\chi_{B(x_n^i,\frac{n-2r}{2})},w(u_n)\chi_{B(x_n^i,\frac{n-2r}{2})})
\\
\end{eqnarray*}
\begin{eqnarray*}
&\geq& -\eps+\liminf_{n\to\infty}\cI(u_n\chi_{B(0,\frac{n-2r}{2})},w(u_n)\chi_{B(0,\frac{n-2r}{2})})\\
&&+\sum_{i=1}^k \liminf_{n\to\infty}\cI_0(u_n(\cdot+x_n^i)\chi_{B(0,\frac{n-2r}{2})},w(u_n)(\cdot+x_n^i)\chi_{B(0,\frac{n-2r}{2})})\\
&\geq& -\eps+\cI(\bar{u}_0,\bar{w}_0)+\sum_{i=1}^k \cI_0(\bar{u}_i,\bar{w}_i)\\
&\geq& -\eps+\cI(\bar{u}_0,w(\bar{u_0}))+\sum_{i=1}^k \cI_0(\bar{u}_i,w_0(\bar{u_i})).
\end{eqnarray*}
Thus taking into account (\ref{Ineq15}) we see that $(h)$ holds and we get
\begin{eqnarray*}
&&\liminf_{n\to\infty}\cI(u_n\chi_{B(0,\frac{n-2r}{2})},w(u_n)\chi_{B(0,\frac{n-2r}{2})})=\cI(\bar{u}_0,w(\bar{u_0})),\\
&&\liminf_{n\to\infty} \cI_0(u_n(\cdot+x_n^i)\chi_{B(0,\frac{n-2r}{2})},w(u_n)(\cdot+x_n^i)\chi_{B(0,\frac{n-2r}{2})})=\cI_0(\bar{u}_i,w_0(\bar{u_i})).
\end{eqnarray*}
Passing to a subsequence if necessary, by Lemma \ref{LemConvWeakIpliesStrong} we obtain
\begin{eqnarray*}\label{EqlocConvW}
&&u_n\chi_{B(0,\frac{n-2r}{2})}+\nabla w(u_n)\chi_{B(0,\frac{n-2r}{2})}\to \bar{u}_0+\nabla w(\bar{u_0})\hbox{ in }L^{p,q},\\
&&u_n(\cdot +x_n^i)\chi_{B(0,\frac{n-2r}{2})}+\nabla w_0(u_n)(\cdot+x_n^i)\chi_{B(0,\frac{n-2r}{2})}\to \bar{u}_i+\nabla w_0(\bar{u_i})\hbox{ in }L^{p,q}. 
\end{eqnarray*}
Therefore by (\ref{Eqstronglok_u_n}) we obtain (\ref{Claim4_1}) and (\ref{Claim4_2}).\\
{\it Claim 6.} $(g)$ holds.\\
From $(c)$ and $(f)$
we know that for any $i\geq 0$
$$u_n(x +x_n^i)\to \bar{u}_i(x),\, \nabla w(u_n)(x)\to \nabla w(\bar{u}_0)(x),\,\nabla w_0(u_n)(x+x_n^i)\to \nabla w_0(\bar{u}_i)(x)\hbox{ a.e. on }\R^3.$$
Replacing $F$ by $\bar{F}$ in Lemma \ref{LemBrezLieb}, where $\bar{F}(x,u)=-V(x)|u|^2+F(x,u)$, $x,u\in\R^3$, we obtain
$$\lim_{n\to\infty}(\cI(u_n,w(u_n))-\cI(u_n-\bar{u}_0,w(u_n)-w(\bar{u}_0)))=\cI(\bar{u}_0,w(\bar{u}_0)).
$$
Thus
$$\lim_{n\to\infty}\cI(u_n,w(u_n))=\cI(\bar{u}_0,w(\bar{u}_0))+\lim_{n\to\infty}\cI(u_n-\bar{u}_0,w(u_n)-w(\bar{u}_0)).$$
Let $E_n:=u_n+\nabla w(u_n)-\bar{u}_0-\nabla w(\bar{u}_0)$. Since the infimum in $|\cdot|_{p,q,1}$ is attained (see \cite{BadialePisaniRolando}[Prop. 2.5]), then there are $E_n^1\in L^p(\R^3,\R^3)$, $E_n^2\in L^{q}(\R^3,\R^3)$ such that $E_n\chi_{B(0,\frac{n-2r}{2})}=E_n^1+E_n^2$ and 
$$|E_n\chi_{B(0,\frac{n-2r}{2})}|_{p,q,1}=|E_n^1|_{p}+|E_n^2|_{q}.$$
Thus $E_n^1\to 0$ in $L^p(\R^3,\R^3)$ and $E_n^2\to 0$ in $L^q(\R^3,\R^3)$.
Observe that 
\begin{eqnarray*}
\int_{\R^3}|V(x)||E_n|^2\, dx&=& \int_{B(0,\frac{n-2r}{2})}|V(x)||E_n|^2\, dx+
\int_{B(0,\frac{n-2r}{2})^c}|V(x)||E_n|^2\, dx\\
&\leq&2\int_{B(0,\frac{n-2r}{2})}V(x)|E_n^1|^2\, dx
+2\int_{B(0,\frac{n-2r}{2})}V(x)|E_n^2|^2\, dx\\
&&+\int_{B(0,\frac{n-2r}{2})^c}V(x)|E_n|^2\, dx\\
&\leq&
2|V|_{\frac{p}{p-2}}|E_n^1|^2_{p}+2|V|_{\frac{q}{q-2}}|E_n^1|^2_{q}\\
&&+
|V\chi_{B(0,\frac{n-2r}{2})^c}|_{\frac{p}{p-2}} |E_n\chi_{\Omega_{E_n}}|^2_{p}+
|V\chi_{B(0,\frac{n-2r}{2})^c}|_{\frac{q}{q-2}}|E_n\chi_{\Omega_{E_n}^c}|^2_{q}
\end{eqnarray*}
Since $E_n$ is bounded in $L^{p,q}$ then
$$\int_{\R^3}|V(x)||E_n|^2\, dx\to 0$$
and thus
\begin{eqnarray*}\label{EqSplitC1}
\lim_{n\to\infty}\cI(u_n,w(u_n))&=&\cI(\bar{u}_0,w(\bar{u}_0))+\lim_{n\to\infty}\cI_0(u_n-\bar{u}_0,w(u_n)-w(\bar{u}_0))\\
&=&
\cI(\bar{u}_0,w(\bar{u}_0))+\lim_{n\to\infty}\cI_0(u_n^0,w_n^0),
\end{eqnarray*}
where $$u_n^j=u_n-\sum_{i=0}^j \bar{u}_i(\cdot -x_n^i)$$
and
$$w_n^j=w(u_n)-w(\bar{u}_0)-\sum_{i=1}^j w_0(\bar{u}_i)(\cdot -x_n^i)$$
for $n\in\N$, $0\leq j< N+1$.
Again by Lemma \ref{LemBrezLieb}
$$\lim_{n\to\infty}(\cI_0(u_n^0(\cdot +x_n^1),w_n^0(\cdot +x_n^1))
-\cI_0(u_n^1(\cdot +x_n^1),w_n^1(\cdot +x_n^1)))=\cI_0(\bar{u}_1,w_0(\bar{u}_1))
$$
and then
\begin{equation*}
\lim_{n\to\infty}\cI_0(u_n^0,w_n^0)=
\cI_0(\bar{u}_1,w_0(\bar{u}_1))+\lim_{n\to\infty}\cI_0(u_n^1,w_n^1).
\end{equation*}
Similarly we show for any $0\leq j<N$
$$\lim_{n\to\infty}(\cI_0(u_n^{j}(\cdot +x_n^{j+1}),w_n^j(\cdot +x_n^{j+1}))
-\cI_0(u_n^{j+1}(\cdot +x_n^{j+1}),w_n^{j+1}(\cdot +x_n^{j+1})))=\cI_0(\bar{u}_{j+1},w_0(\bar{u}_{j+1}))
$$
and then
\begin{equation*}
\lim_{n\to\infty}\cI_0(u_n^j,w_n^j)=
\cI_0(\bar{u}_{j+1},w_0(\bar{u}_{j+1}))+\lim_{n\to\infty}\cI_0(u_n^{j+1},w_n^{j+1}).
\end{equation*}
Thus we obtain
$$\lim_{n\to\infty}\cI(u_n,w(u_n))=\cI(\bar{u}_0,w(\bar{u}_0))+\sum_{i=1}^{j+1}\cI_0(\bar{u}_i,w(\bar{u}_i))+\lim_{n\to\infty}\cI_0(u_n^{j+1},w_n^{j+1})$$
for any  $0\leq j<N$.
If $N<\infty$ then owing to $(h)$ we get 
$$\lim_{n\to\infty}\cI_0(u_n^{N},w_n^{N})=0.$$
By $(d)$ we have $u_n^{N}\to 0$ in $L^{p,q}$. Hence by Lemma \ref{LemEstimate} $(b)$ we get $w_n^{N}\to 0$ in $L^{p,q}$ as well.
If $N=\infty$ then
$$\lim_{n\to\infty}\cI_0(u_n^{n},w_n^{n})=0$$
and thus $w_n^{n}\to 0$ in $L^{p,q}$ and we get $(g)$.
\end{proof}

\begin{altproof}{Theorem \ref{ThMainSplitting}}
Observe that by Lemma \ref{DefofW(u)} and Proposition \ref{PropDefOfm(u)}, if $(E_n)_{n=0}^{\infty}\subset \cM$ then
$E_n=u_n+\nabla w(u_n)$ for some $(u_n,w(u_n))\in \cN$. In view of Lemma \ref{LemSplit1} we get a sequence $(\bar{u}_i,w(\bar{u}_i))$, hence $\bar{E}_i:=\bar{u}_i+\nabla w(\bar{u}_i)\in  \D(\curl,p,q)\setminus\{0\}$ for $i\geq 1$. Moreover \eqref{EqThMainSplitting1}, \eqref{EqThMainSplitting2} and \eqref{eqInfinitesplitting} follows from Lemma \ref{LemSplit1} (b) - (h).
\end{altproof}

In general $\J'$ is not (sequentially) weak-to-weak$^*$ continuous. Indeed, take e.g. 
$F(x,u)=\frac1p((1+|u|^q)^{\frac{p}{q}}-1),$ and observe that
$\nabla w_n\weakto \nabla w$ in $L^{p,q}$ does not imply
$$(1+|\nabla w_n|^q)^{\frac{p}{q}}|\nabla w_n|^{q-2}(\nabla w_n)\weakto (1+|\nabla w|^q)^{\frac{p}{q}}|\nabla w|^{q-2}(\nabla w)$$ in $(L^{p,q})^*=L^{\frac{p}{p-1}}(\R^3,\R^3)\cap L^{\frac{q}{q-1}}(\R^3,\R^3)$. 
However we show the weak-to-weak$^*$ continuity of $\J'$ for sequences on the Nehari-Pankov manifold $\mathcal{N}$. Obviously the same regularity holds for $\cE$ and $\cM$.

\begin{Cor}\label{CorJweaklycont}
If $(u_n,w_n)\in\mathcal{N}$ and $(u_n,w_n)\rightharpoonup (u_0,w_0)$ in $\U\times\W$ then $\J'(u_n,w_n)\rightharpoonup \J'(u_0,w_0)$, i.e.
$$\J'(u_n,w_n)(\phi,\psi)\to \J'(u_0,w_0)(\phi,\psi)$$
for any $(\phi,\psi)\in\U\times\W$.
\end{Cor}
\begin{proof}
Observe that by Proposition \ref{Propuv_N}, Proposition \ref{PropDefOfm(u)} $(a)$ and Lemma \ref{LemDefofW} $(c)$ we get $w_n=w(u_n)$. In view of Lemma \ref{LemSplit1} $(c)$ and $(f)$ we may assume that  $u_n+\nabla w_n\to u_0+\nabla w_0$ a.e. on $\R^3$.
Observe that for $(\phi,\psi)\in\U\times\W$
\begin{eqnarray*}
\J'(u_n,w_n)(\phi,\psi)-\J'(u_0,w_0)(\phi,\psi)&=&\int_{\R^3}\langle \nabla u_n-\nabla u_0,\nabla \phi\rangle\,dx\\
&&+\int_{\R^3}V(x)\langle u_n+\nabla w_n-u_0-\nabla w_0,\phi+\nabla \psi\rangle\,dx\\
&&-\int_{\R^3}\langle f(x,u_n+\nabla w_n)-f(x,u_0+\nabla w_0),\phi+\nabla \psi\rangle\,dx. 
\end{eqnarray*}
In view of the Vitaly convergence theorem we obtain
$$\J'(u_n,w_n)(\phi,\psi)-\J'(u_0,w_0)(\phi,\psi)\to 0.$$
\end{proof}

\section{Analysis of Palais-Smale sequences in $\cN$}\label{SectAnalysisPS}

The following lemma implies that any Palais-Smale sequence of $\J$ in $\cN$ 
is bounded.

\begin{Lem}\label{LemCoercive}
$\J$ is coercive on $\mathcal{N}$.
\end{Lem}
\begin{proof}
Suppose that $(u_n,w_n)\in\mathcal{N}$, $\|(u_n,w_n)\|\to\infty$ as $n\to\infty$ and  $\J(u_n,w_n)\leq M$ for some constant $M>0$. Let 
$$\bar{u}_n:=\frac{u_n}{\|(u_n,w_n)\|}.$$
In view of Lemma \ref{LemSplit1} $(c)$ we may assume that $\bar{u}_n\rightharpoonup \bar{u}_0$ in $\U$ and $\bar{u}_n\to \bar{u}_0$ a.e. in $\R^3$. Moreover there is a sequence $(x_n)_{n\in\N}\subset\R^3$ such that
\begin{equation}\label{EqLemLions}
\liminf_{n\to\infty}\int_{B(x_n,1)}|\bar{u}_n|^2\, dx >0.
\end{equation}
Otherwise, in view of \cite{DAprileSiciliano}[Lemma 4.1]) we get that $\bar{u}_n\to 0$ in $L^{p,q}$. By the continuity of $I_0$ 
$$\int_{\R^N}F(x,s\bar{u}_n)\,dx\to 0$$ 
for any $s\geq0$. Let us fix $s\geq0$. By Proposition \ref{Propuv_N}
\begin{equation}\label{EqIneq2}
M\geq \limsup_{n\to\infty}\J(u_n,w_n)\geq \limsup_{n\to\infty} \J(s\bar{u}_n,0)= \frac{s^2}{2}\limsup_{n\to\infty}\|\bar{u}_n\|^2_{\D}.
\end{equation}
In view of Lemma \ref{LemEstimate} $(b)$ and Proposition \ref{PropDefOfm(u)} $(b)$ we have
$$\frac{1}{2}\|u_n\|^2_\D-c_1\min\{|u_n+\nabla w_n|_{p,q}^p,|u_n+\nabla w_n|_{p,q}^q\}\geq \J(u_n,w_n)\geq c:=\inf_{\mathcal{N}}\J>0.$$
Moreover by Lemma \ref{HelmholzDec} there are continuous projections of $\cl\U\oplus\nabla\W$ onto $\nabla\W$ and onto $\U$ in $L^{p,q}$. Hence there is a constant $C_1\in(0,1)$ such that
\begin{eqnarray}\label{ineqPS1}
C_1|\nabla w_n|_{p,q}\leq |u_n+\nabla w_n|_{p,q},\\
C_1|u_n|_{p,q}\leq |u_n+\nabla w_n|_{p,q}\label{ineqPS2}
\end{eqnarray}
for every $n$. Then 
\begin{eqnarray*}
2\|u_n\|^2_\D&\geq& \|u_n\|^2_\D+2c+2c_1\min\{|u_n+\nabla w_n|_{p,q}^p,|u_n+\nabla w_n|_{p,q}^q\}\\
&\geq& \|u_n\|^2_\D+2c+2c_1C_1^q\min\{|\nabla w_n|_{p,q}^p,|\nabla w_n|_{p,q}^q\}
\end{eqnarray*}
If $\liminf_{n\to\infty}|\nabla w_n|_{p,q}=0$ then, up to a subsequence, $|\nabla w_n|_{p,q}\to 0$, and for sufficiently large $n$ we get
\begin{eqnarray*}
2\|u_n\|^2_\D&\geq& 
 \|u_n\|^2_\D+|\nabla w_n|_{p,q}^2=\|(u_n,w_n)\|^2.
\end{eqnarray*}
If $\liminf_{n\to\infty}|\nabla w_n|_{p,q}>0$ then there is  $C_2\in (0,1)$ 
such that for sufficiently large $n$
\begin{eqnarray*}
2\|u_n\|^2_\D
&\geq& C_2(\|u_n\|^2_\D+|\nabla w_n|_{p,q}^2)=C_2\|(u_n,w_n)\|^2.
\end{eqnarray*}
Therefore, passing to a subsequence if necessary, 
$$\inf_{n\in\N}\|\bar{u}_n\|^2_\D=\inf_{n\in\N}\frac{\|u_n\|_\D^2}{\|(u_n,w_n)\|^2}>0$$
and by (\ref{EqIneq2})
$$M\geq \frac{s^2}{2}\inf_{n\in\N}\|\bar{u}_n\|^2_\D$$
for any $s\geq0$. The obtained contradiction shows that (\ref{EqLemLions}) holds. Then we may assume that $(x_n)\subset\Z^3$ and
$$\liminf_{n\to\infty}\int_{B(0,r)}|\bar{u}_n(x+x_n)|^2\, dx >0$$
for some $r>1$, hence $\bar{u}_n(\cdot+x_n)\to \bar{u}_0$ in $L^2_{loc}(\R^N)$ for some $\bar{u}_0\neq 0$. Take any bounded $\Om\subset\R^3$ of positive measure such that
$$\Om\subset \{x\in\R^3|\;\bar{u}_0(x)\neq 0\}.$$
Observe that for any $x\in\Om$
$$|u_n(x+x_n)|=|\bar{u}_n(x+x_n)|\cdot \|(u_n,w_n)\|\to\infty$$
and by Fatou's lemma
\begin{equation}\label{ineqPS3}
\int_{\Om}\frac{|u_n(x+x_n)|^p}{\|(u_n,w_n)\|^2}\,dx=
\int_{\Om}|u_n(x+x_n)|^{p-2}|\bar{u}_n(x+x_n)|^2\,dx\to\infty
\end{equation}
as $n\to\infty$.
Since norms $|\cdot|_{p,q}$ and $|\cdot|_{p}$ are equivalent on $L^p(\Om,\R^3)$ (see \cite{BadialePisaniRolando}[Corollary 2.15]), then  
the periodicity of $F$ in $x$, Lemma \ref{LemEstimate} $(b)$ and \eqref{ineqPS2}  imply
\begin{eqnarray*}
\frac{\J(u_n,w_n)}{\|(u_n,w_n)\|^2}&\leq&\frac{1}{2}\|\bar{u}_n\|^2_\D
-\frac{\int_{\R^3}F(x,u_n(x+x_n)+\nabla w_n(x+x_n))\,dx}{\|(u_n,w_n)\|^2}\\
&\leq&\frac{1}{2}\|\bar{u}_n\|^2_\D
-c_1 \frac{\min\{C_1^p|u_n(\cdot+x_n)\chi_\Om|_{p,q}^p, C_1^q|u_n(\cdot+x_n)\chi_\Om|_{p,q}^q\}}{\|(u_n,w_n)\|^2}\\
&\leq&\frac{1}{2}\|\bar{u}_n\|^2_\D
-C_3 \min\Big\{\frac{|u_n(\cdot+x_n)\chi_\Om|_{p}^p}{\|(u_n,w_n)\|^2}, \frac{|u_n(\cdot+x_n)\chi_\Om|_{p}^q}{\|(u_n,w_n)\|^2}\Big\}
\end{eqnarray*}
fore some constant $C_3>0$.
Thus by \eqref{ineqPS3} we get 
$$\frac{\J(u_n,w_n)}{\|(u_n,w_n)\|^2}\to\infty$$
as $n\to\infty$ and the obtained contradiction completes proof.
\end{proof}

\begin{Lem}\label{LemV_ynto0}
If $E\in L^{p,q}$ and $x_n\in\R^3$ is such that
$|x_n|\to+\infty$ as $n\to +\infty$, then
$$\lim_{n\to\infty}\int_{\R^3} V(x+x_n)|E|^2\, dx=0.$$
\end{Lem}
\begin{proof}
Observe that
\begin{eqnarray*}
\int_{\R^3}|V(x+x_n)||E|^2\, dx
&=&\int_{B(0,R)}|V(x+x_n)||E|^2\, dx
+\int_{B(0,R)^c}|V(x+x_n)||E|^2\, dx\\
&\leq&
\Big(\int_{B(x_n,R)}|V(x)|^{\frac{q}{q-2}}\, dx\Big)^{\frac{q-2}{q}} |E\chi_{\Omega_{E}^c}|^2_{q}\\
&&+\Big(\int_{B(x_n,R)}|V(x)|^{\frac{p}{p-2}}\, dx\Big)^{\frac{p-2}{p}} |E\chi_{\Omega_{E}}|^2_{p}\\
&&+
|V|_{\frac{q}{q-2}}|E\chi_{\Omega_{E}^c\cap {B(0,R)^c}}|^2_{q}
+|V|_{\frac{p}{p-2}} |E\chi_{\Omega_{E\cap {B(0,R)^c}}}|^2_{p}.
\end{eqnarray*}
for any $R>0$.
Therefore
$$\lim_{n\to\infty}
\int_{\R^3}|V(x+x_n)||E|\, dx\leq 
(|V|_{\frac{q}{q-2}}+|V|_{\frac{p}{p-2}})(|E\chi_{\Omega_{E\cap {B(0,R)^c}}}|^2_{q}+|E\chi_{\Omega_{E\cap {B(0,R)^c}}}|^2_{p})$$
and we get the conclusion by taking $R\to+\infty$.
\end{proof}

\begin{Lem}\label{LemSplit2}
Let 
$\J_0:\U\times\W\to\R$ be the functional given by
\begin{equation}\label{eqJ_0}
\J_0(u,w)=\frac{1}{2}\int_{\R^3}|\nabla u|^2\,dx 
-\int_{\R^3}F(x,u+\nabla w)\,dx.
\end{equation}
for $(u,w)\in\U\times\W$.
Let $(u_n,w_n)\in\mathcal{N}$ be a $(PS)_c$-sequence for some $c>0$.
Then there is $N\geq 0$ and there are sequences 
$(\bar{u}_i,\bar{w}_i)_{i=0}^N\subset\U\times \W$ and $(x_n^i)_{0\leq i\leq N,n\geq i}\subset \Z^3$ such that $x_n^0=0$ and, up to a sequence,
\begin{eqnarray}\label{EqSplit1}
&& \J'(\bar{u}_0,\bar{w}_0)=0,\\\label{EqSplit2}
&& \J'_0(\bar{u}_i,\bar{w}_i)=0\text{ for } i=1,...,N,\\\label{EqSplit3}
&& \bar{u}_i\neq 0\text{ for } i=1,...,N,\\\label{EqSplit4}
&& u_n-\sum_{i=0}^N \bar{u}_i(\cdot -x_n^i)\to0 \hbox{ in }\D(\R^3,\R^3)
\\\label{EqSplit5}
&& w_n-\sum_{i=0}^N \bar{w}_i(\cdot -x_n^i)\to0 \hbox{ in }\W
\\\label{EqSplit6}
&& \J(\bar{u}_0,\bar{w}_0)+\sum_{i=1}^N
\J_0(\bar{u}_i,\bar{w}_i)=c.
\end{eqnarray}
\end{Lem}
\begin{proof}$\hbox{}$\\ 
{\it Step 1.} Construction of $(\bar{u}_i,\bar{w}_i)$, $(x_n^i)_{n\geq i}$ and proof of (\ref{EqSplit1}).\\
Since $(u_n,w_n)\in \mathcal{N}$ then by Proposition \ref{Propuv_N}, Proposition \ref{PropDefOfm(u)} $(a)$ and Lemma \ref{LemDefofW} 
$$m(u_n)=(u_n,w_n)\hbox{ and }w_n=w(u_n).$$
In view of Lemma \ref{LemCoercive} $(u_n,w_n)$ is bounded in $\U \times \W$. Thus we may assume that
$$u_n\rightharpoonup \bar{u}_0\hbox{ in }\D(\R^3,\R^3)\hbox{ and }\;\nabla w_n\rightharpoonup \nabla \bar{w}_0\hbox{ in }L^{p,q}.$$
In view of Lemma \ref{LemSplit1} there is $N\in\N\cup\{\infty\}$ and there exist sequences $(\bar{u}_i)_{i\in\N}\subset \D(\R^3,\R^3)$ and $(x_n^i)_{n\geq i}\subset \Z^3$ such that $x_n^0=0$ and, up to a subsequence, $(a)$ -- $(h)$ are satisfied.
We take
$$\bar{w}_0:=w_0(\bar{u}_0)$$
and
$$\bar{w}_i:=w_0(\bar{u}_i)$$
for $i\geq 1$.
In view of Corollary \ref{CorJweaklycont}
$$\J'(\bar{u}_0,\bar{w}_0)=0.$$
{\it Step 2.} $\J_0'(\bar{u}_i,\bar{w}_i)=0$ for $1\leq i<N+1$.\\ 
From $(b)$ and $(e)$ of Lemma \ref{LemSplit1} and arguing as in Corollary \ref{CorJweaklycont} we obtain
$$\J_0'(u_n(\cdot+x_n^i),w_n(\cdot+x_n^i))(\phi,\psi)\to \J_0'(\bar{u}_i,\bar{w}_i)(\phi,\psi)$$
for any $(\phi,\psi)\in \U\times\W$.
On the other hand 
\begin{eqnarray*}
&&|\J_0'(u_n(\cdot+x_n^i),w_n(\cdot+x_n^i))(\phi,\psi)|\leq
|\J'(u_n,w_n)(\phi(\cdot -x_n^i),\psi(\cdot -x_n^i))|\\
&&+
\int_{\R^3}|V(x)||\langle u_n+\nabla w_n,
\phi(\cdot -x_n^i)+\nabla \psi(\cdot -x_n^i)\rangle|\, dx\\
&&\leq
|\J'(u_n,w_n)(\phi(\cdot -x_n^i),\psi(\cdot -x_n^i))|\\
&&+\Big(\int_{\R^3}|V(x)||u_n+\nabla w_n|^2\,dx\Big)^{\frac{1}{2}}\cdot
\Big(\int_{\R^3}|V(x+y_n)||\phi+\nabla \psi|^2\,dx\Big)^{\frac{1}{2}}
\end{eqnarray*}
and by Lemma \ref{LemV_ynto0} we get 
$$\J_0'(u_n(\cdot+x_n^i),w_n(\cdot+x_n^i))(\psi,\phi)\to 0$$
for any $(\phi,\psi)\in \U\times\W$.
Hence
$$\J_0'(\bar{u}_i,\bar{w}_i)=0.$$
{\it Step 3.} $\inf_{1\leq i <N+1}|\bar{u}_{i}|_{p,q}>0$.\\
If $N<\infty$ then we conclude directly from Lemma \ref{LemSplit1} $(a)$. Assume that $N=\infty$ and let $i\geq 1$.
Similarly as in proof of Proposition \ref{PropDefOfm(u)} $(a)$ (see (A4)) we get 
$$\inf_{\|u\|_{\D}=r}\J_0(u,0)>0$$
for sufficiently small $r>0$. 
Since $\J_0'(\bar{u}_{i},\bar{w}_{i})=0$ and $\bar{u}_i\neq 0$
then $(\bar{u}_{i},\bar{w}_{i})\in\cN_0$, where $\mathcal{N}_0$ is given by $(\ref{DefOfN})$ under assumption $V=0$.
Assuming that $V=0$ in Proposition \ref{Propuv_N}  we show that
$$\J_0(\bar{u}_{i},\bar{w}_{i})\geq \J_0(t\bar{u}_{i},0)$$
for any $t\geq 0$.
Thus
\begin{equation}\label{EqJ_0geqc}
\J_0(\bar{u}_{i},\bar{w}_{i})\geq \J_0\Big(\frac{r}{\|\bar{u}_{i}\|_{\D}}\bar{u}_{i},0\Big)\geq \inf_{\|u\|_{\D}=r}\J_0(u,0)>0.
\end{equation}
Note that by (\ref{EqSeries}) $(\bar{u}_i+\nabla \bar{w}_i)_{i\geq 1}$ is bounded and
if, up to a subsequence $\bar{u}_i\to 0$ in $L^{p,q}$, 
then
\begin{eqnarray*}
\|\bar{u}_i\|_{\D}^2&=&\J'_0(\bar{u}_i,\bar{w}_i)(\bar{u}_i,0)
+\int_{\R^3}\langle f(x,\bar{u}_i+\nabla \bar{w}_i),\bar{u}_i\rangle\, dx
=\int_{\R^3}\langle f(x,\bar{u}_i+\nabla \bar{w}_i),\bar{u}_i\rangle\, dx\to 0
\end{eqnarray*}
as $i\to\infty$.
Furthermore
\begin{eqnarray*}
\limsup_{n\to\infty}\J_0(\bar{u}_{i},\bar{w}_{i})&=&
\limsup_{n\to\infty}\Big(-\int_{\R^3}F(x,\bar{u}_{i}+\nabla \bar{w}_{i})\,dx\Big)\leq 0
\end{eqnarray*}
which contradicts (\ref{EqJ_0geqc}). Therefore 
$$\inf_{i\geq 1}|\bar{u}_{i}|_{p,q}>0.$$
{\it Step 4.} $N<\infty$ and proof of (\ref{EqSplit2}), (\ref{EqSplit3}) and (\ref{EqSplit5}).\\
Observe that for some constant $C_1>0$ and for any $k\geq 1$ 
$$C_1\sum_{i=1}^k|\bar{u}_i|^6_{p,q}\leq 
\sum_{i=1}^k|\bar{u}_i|^6_{6}\leq 
\sum_{i=1}^k\liminf_{n\to\infty}|u_n(\cdot +x_n^i)\chi_{B(0,\frac{n-2r}{2})}|^6_{6}\leq \liminf_{n\to\infty}|u_n|_6^6$$
where the last inequalities follows from the fact that $B(x_n^i,\frac{n-2r}{2})\cap B(x_n^j,\frac{n-2r}{2})\neq \emptyset$ if $i\neq j$. Since $(u_n)$ is bounded in $L^6(\R^3,\R^3)$ and
taking into account {\it Step 3}
we obtain that $\bar{u}_i\neq 0$ for finitely many $i\geq 1$. Thus $N<\infty$ and (\ref{EqSplit2}), (\ref{EqSplit3}), (\ref{EqSplit5}) follow from {\it Step 2}, {\it Step 3} and Lemma \ref{LemSplit1} $(g)$.\\
{\it Step 5.} Proof of (\ref{EqSplit4}).\\
Let $v_n:=\sum_{i=0}^N \bar{u}_i(\cdot -x_n^i)$ and note that
$u_n-v_n\rightharpoonup 0$ in $\D(\R^3,\R^3)$ and $u_n-v_n\to 0$ in $L^{p,q}$.
Since
\begin{eqnarray*}
&&\J'(u_n,w_n)(u_n-v_n,0)=\|u_n-v_n\|_{\D}^2+\int_{\R^3}\langle \nabla v_n,\nabla u_n-\nabla v_n\rangle\, dx\\
&&+\int_{\R^3}V(x)\langle u_n+\nabla w_n, u_n-v_n\rangle \, dx
-\int_{\R^3}\langle f(x,u_n+\nabla w_n),u_n-v_n\rangle\, dx
\end{eqnarray*}
then $\|u_n-v_n\|_{\D}\to 0$.\\
{\it Step 6.} Proof of (\ref{EqSplit6}).\\
Since $N<\infty$ and Lemma \ref{LemSplit1} $(h)$ holds, then we need to prove the following convergence
\begin{equation}
\lim_{n\to\infty}\|u_n\|^2_{\D}= \sum_{i=0}^N\|\bar{u}_i\|^2_{\D}.
\end{equation}
Note that
\begin{eqnarray*}
&&\Big\|\sum_{i=0}^N\bar{u}_i(\cdot-x_n^i)\Big\|^2_{\D}=
\sum_{i=0}^N\|\bar{u}_i\|^2_{\D}+2\sum_{1\leq i<j\leq N}\int_{\R^3}\langle \bar{u}_i(\cdot-x_n^i),\bar{u}_j(\cdot-x_n^j)\rangle\,dx
\end{eqnarray*}
and
\begin{eqnarray*}
\int_{\R^3}|\langle \bar{u}_i(\cdot-x_n^i),\bar{u}_j(\cdot-x_n^j)\rangle|\,dx&=&\int_{B(0,R)}|\langle \bar{u}_i,\bar{u}_j(\cdot+x_n^i-x_n^j)\rangle|\,dx\\
&&+\int_{\R^3\setminus B(0,R)}|\langle \bar{u}_i,\bar{u}_j(\cdot+x_n^i-x_n^j)\rangle|\,dx\\
&\leq& \|\bar{u}_i\|_{\D}\|\bar{u}_j\chi_{B(x_n^i-x_n^j,R)}\|_{\D}+
\|\bar{u}_i\chi_{\R^3\setminus B(0,R)}\|_{\D}\|\bar{u}_j\|_{\D}
\end{eqnarray*}
for any $R>0$. If $i<j$ then 
$$\limsup_{n\to\infty}\int_{\R^3}|\langle \bar{u}_i(\cdot-x_n^i),\bar{u}_j(\cdot-x_n^j)\rangle|\,dx
\leq \|\bar{u}_i\chi_{\R^3\setminus B(0,R)}\|_{\D}\|\bar{u}_j\|_{\D}.$$
 If $R\to \infty$ then we obtain
$$\lim_{n\to\infty}\|u_n\|^2_{\D}=\lim_{n\to\infty}\Big\|\sum_{i=0}^N\bar{u}_i(\cdot-x_n^i)\Big\|^2_{\D}=
\sum_{i=0}^N\|\bar{u}_i\|^2_{\D}.$$
\end{proof}

\begin{altproof}{Theorem \ref{ThMainPS_Splitting}}
Proof follows directly from Lemma \ref{LemSplit2} by decomposing $E_n=u_n+\nabla w_n$, where $(u_n,w_n)\in\cN$ and by taking $\bar{E}_i=\bar{u}_i+\nabla \bar{w}_i$ for $0\leq i \leq N$.
\end{altproof}

\section{Proofs of Theorem \ref{ThMain} and Theorem \ref{ThPohozaev}}
\label{SectProofs}

Now we are ready to prove the existence and nonexistence results.

\begin{Prop}\label{PropCrit2}
There is a critical point $(u_0,w_0)\in\mathcal{N}_0$ of $\J_0$ such that $u_0\neq 0$ and
\begin{equation}\label{DefOfCritic0}
\J_0(u_0,w_0)=c_0:=\inf_{\mathcal{N}_0}\J_0>0.
\end{equation}
\end{Prop}
\begin{proof}
In view of Proposition \ref{PropDefOfm(u)} $(b)$ there is
$u_n\in S_\U$ such that $\J_0(m_0(u_n))\to c_0>0$ and 
$\J_0'(m_0(u_n))\to 0$, where $m_0$ is given in Proposition \ref{PropDefOfm(u)} $(a)$ under assumption $V=0$. Then by Lemma \ref{LemSplit2} condition (\ref{EqSplit6}) holds. Thus $N=0$ or $N=1$. If $N=0$, then $(u_0,w_0):=(\bar{u}_0,\bar{w}_0)$ is a critical point of $\J_0$, $(u_0,w_0)\in\mathcal{N}_0$ and $u_0\neq 0$. Similarly we obtain a nontrivial critical point in case $N=1$.
\end{proof}

\begin{Prop}\label{PropCrit3}
There is a critical point $(u,w)$ of $\J$ such that $u\neq 0$. If
\begin{equation}\label{GroundCond}
\int_{\R^3}V(x)|u_0+\nabla w_0|^2\, dx < 0,
\end{equation}
where $(u_0,w_0)$ is given in Proposition \ref{PropCrit2},
then $(u,w)\in\mathcal{N}$ and
\begin{equation*}\label{DefOfCritic1}
\J_0(u_0,w_0)>\J(u,w)=c:=\inf_{\mathcal{N}}\J>0.
\end{equation*}
\end{Prop}
\begin{proof}
Let (\ref{GroundCond}) hold.
Observe that by Proposition \ref{PropCrit2} and Proposition \ref{Propuv_N} we have
$$c_0=\J_0(u_0,w_0)\geq \J_0(m(u_0))>\J(m(u_0))\geq c.$$
Note that any critical point $(\bar{u},\bar{w})$ of $\J_0$ such that $\bar{u}\neq 0$ belongs to $\mathcal{N}_0$ and hence 
$$\J_0(\bar{u},\bar{w})\geq c_0>0.$$ In view of Proposition \ref{PropDefOfm(u)} $(b)$ there is a $(PS)_c$-sequence $(u_n,w_n)\in\mathcal{N}$.
Therefore by Lemma \ref{LemSplit2} condition (\ref{EqSplit6}) implies that
$N=0$ and from (\ref{EqSplit4}), (\ref{EqSplit5}) we have $u_n\to \bar{u}_0$ in $\U$ and $w_n\to\bar{w}_0$ in $\W$. Thus $(u,w):=(\bar{u}_0,\bar{w}_0)$ is a critical point of $\J$ such that $\J(u,w)=c>0$ and $u\neq 0$.
Suppose that $$\int_{\R^3}V(x)|u_0+\nabla w_0|^2\,dx =0$$
then $V(x)|u_0(x)+\nabla w_0(x)|^2=0$ a.e. on $\R^3$. Then we easily see that $\J(u_0,w_0)=\J_0(u_0,w_0)$ and $\J'(u_0,w_0)=\J'_0(u_0,w_0)$ and by Proposition \ref{PropCrit2} $(u,w):=(u_0,w_0)$ is a critical point of $\J$ such that $u\neq 0$.
\end{proof}

\begin{altproof}{Theorem \ref{ThMain}}
Proof follows directly from Proposition \ref{PropCrit2}, Proposition \ref{PropCrit3} and Proposition \ref{PropSolutE}.
\end{altproof}

\begin{altproof}{Theorem \ref{ThPohozaev}}
Let $E=u+\nabla w$ be a classical solution of \eqref{eq}, i.e.
\begin{equation}\label{EqPoh}
\curlop\curlop E = f(E)\hbox{ in }\R^3
\end{equation}
such that $\div(u)=0$ and \eqref{EqThPohozaev1}, \eqref{EqThPohozaev2} holds.
Let $\vp\in C^\infty_0(\R)$ be such that $0\leq \vp\leq 1$, $\vp(r)=1$ for $r\leq 1$ and $\vp(r)=0$ for $r\geq 2$. Similarly as in \cite{Willem}[Theorem B.3.] we define $\vp_n\in C^\infty_0(\R^3)$ by the following formula
$$\vp_n(x)=\vp\Big(\frac{|x|^2}{n^2}\Big).$$ 
Then there exists $C\geq 0$ such that
$$\vp_n(x)\leq C\hbox{ and }|x||\nabla \vp_n(x)|\leq C$$
for every $n$ and $x\in\R^3$. Recall that (see \cite{Willem})  
\begin{eqnarray*}
\Delta u_i \vp_n\langle x, \nabla u_i\rangle &=& \div\Big(\vp_n(\nabla u_i\langle x,\nabla u_i\rangle-x\frac{|\nabla u_i|^2}{2})\Big)+\frac{1}{2}\vp_n|\nabla u_i|^2\\
&&-\langle \nabla\vp_n,\nabla u_i\rangle\langle x,\nabla u_i\rangle+\langle\nabla\vp_n,x\rangle\frac{|\nabla u_i|^2}{2}
\end{eqnarray*}
for $i=1,2,3$.
Since $\supp(\vp_n)\subset \Om_n:=B(0,3n^2)$, then by the divergence theorem
\begin{eqnarray*}
\int_{\Om_n} \Delta u_i \vp_n\langle x, \nabla u_i\rangle\, dx &=&
\frac{1}{2}\int_{\Om_n}\vp_n|\nabla u_i|^2\,dx\\
&&+\int_{\Om_n}-\langle \nabla\vp_n,\nabla u_i\rangle\langle x,\nabla u_i\rangle+\langle\nabla\vp_n,x\rangle\frac{|\nabla u_i|^2}{2}\, dx.
\end{eqnarray*}
Hence
\begin{eqnarray}
\label{eqPoh2}
\int_{\R^3} \Delta u_i \vp_n\langle x, \nabla u_i\rangle\, dx &=&
\frac{1}{2}\int_{\R^3}\vp_n|\nabla u_i|^2\,dx\\
\nonumber
&&+\int_{\R^3}-\langle \nabla\vp_n,\nabla u_i\rangle\langle x,\nabla u_i\rangle+\langle\nabla\vp_n,x\rangle\frac{|\nabla u_i|^2}{2}\, dx.
\end{eqnarray}
Observe that
\begin{eqnarray*}
\div(x\vp_n F(E))&=&
3\vp_nF(E)+\langle  f(E),\vp_n\sum_{i=1}^3x_i\partial_{x_i} E \rangle+\langle\nabla\vp_n,x\rangle F(E)
\end{eqnarray*}
and again by the divergence theorem
\begin{eqnarray}
\label{eqPoh3}
\int_{\R^3} \langle  f(E),\vp_n\sum_{i=1}^3x_i\partial_{x_i} u \rangle\, dx
&=& -\int_{\R^3} \langle  f(E),\vp_n\sum_{i=1}^3x_i\partial_{x_i} \nabla w \rangle\, dx\\
\nonumber
&& -3\int_{\R^3}\vp_nF(E)\, dx -\int_{\R^3} \langle\nabla\vp_n,x\rangle F(E)\, dx.
\end{eqnarray}
Multiplying \eqref{EqPoh} by $\vp_n\sum_{i=1}^3x_i\partial_{x_i} u $ and integrating over $\R^3$ we get
\begin{eqnarray*}
\int_{\R^3} \langle  f(E),\vp_n\sum_{i=1}^3x_i\partial_{x_i} u \rangle\, dx
&=& \int_{\R^3} \langle\curlop\curlop E, \vp_n\sum_{i=1}^3x_i\partial_{x_i} u \rangle\, dx\\
&=& \int_{\R^3} \langle\curlop\curlop u, \vp_n\sum_{i=1}^3x_i\partial_{x_i} u \rangle\, dx\\
&=& \int_{\R^3} \langle -\Delta u, \vp_n\sum_{i=1}^3x_i\partial_{x_i} u \rangle\, dx.
\end{eqnarray*}
Therefore in view of \eqref{eqPoh2} and \eqref{eqPoh3} we obtain 
\begin{eqnarray}
\label{EqPoh4}
&&\int_{\R^3} \langle  f(E),\vp_n\sum_{i=1}^3x_i\partial_{x_i} \nabla w \rangle\, dx
+3\int_{\R^3}\vp_nF(E)\, dx +\int_{\R^3} \langle\nabla\vp_n,x\rangle F(E)\, dx\\
\nonumber
&&=\frac{1}{2}\int_{\R^3}\vp_n|\nabla u|^2\,dx
+\sum_{i=1}^3\int_{\R^3}-\langle \nabla\vp_n,\nabla u_i\rangle\langle x,\nabla u_i\rangle+\langle\nabla\vp_n,x\rangle\frac{|\nabla u_i|^2}{2}\, dx.
\end{eqnarray}
By direct computations we show that
\begin{eqnarray*}
\nabla \big(\vp_n(\langle x, \nabla w\rangle -w)\big)  
&=&\vp_n(\langle x, \partial_{x_1}(\nabla w)\rangle,\langle x, \partial_{x_2}(\nabla w)\rangle,\langle x, \partial_{x_3}(\nabla w)\rangle)\\
&&+ \nabla \vp_n (\langle x, \nabla w\rangle -w)
\end{eqnarray*}
and
\begin{eqnarray*}
\langle  f(E),\vp_n\sum_{i=1}^3x_i \partial_{x_i}\nabla  w \rangle
&=&\langle  f(E),\nabla \big(\vp_n(\langle x, \nabla w\rangle -w)\big)\rangle
\\
&& -\langle  f(E),\nabla \vp_n (\langle x, \nabla w\rangle -w)\rangle.
\end{eqnarray*}
Multiplying \eqref{EqPoh} by $\nabla \big(\vp_n(\langle x, \nabla w\rangle -w)\big)$ and integrating over $\R^3$ we get
$$\int_{\R^3} \langle  f(E),\nabla \big(\vp_n(\langle x, \nabla w\rangle -w)\big)\rangle\, dx=0,$$
and thus \eqref{EqPoh4} takes the following form
\begin{eqnarray*}
&&-\int_{\R^3} \langle  f(E),\nabla \vp_n (\langle x, \nabla w\rangle -w)\rangle\, dx
+3\int_{\R^3}\vp_nF(E)\, dx +\int_{\R^3} \langle\nabla\vp_n,x\rangle F(E)\, dx\\
\nonumber
&&=\frac{1}{2}\int_{\R^3}\vp_n|\nabla u|^2\,dx
+\sum_{i=1}^3\int_{\R^3}-\langle \nabla\vp_n,\nabla u_i\rangle\langle x,\nabla u_i\rangle+\langle\nabla\vp_n,x\rangle\frac{|\nabla u_i|^2}{2}\, dx.
\end{eqnarray*}
Since $\nabla \vp_n (x)=0$ for $|x|< n^2$, then by the Lebesgue dominated theorem we get 
\begin{eqnarray*}
3\int_{\R^3} F(E)\, dx &=&
\frac{1}{2}\int_{\R^3}|\nabla u|^2\,dx
=\frac{1}{2}\int_{\R^3}|\curlop E|^2\,dx
\end{eqnarray*}
which completes the proof.
\end{altproof}

\begin{altproof}{Corollary \ref{CorollaryMain}}
Suppose that $V=0$ and $E=u+\nabla w$ is a classical solution to \eqref{EqPoh} with $u\neq 0$ and $p>6$ or $q<6$. Then by \eqref{EqThPohozaev}
\begin{eqnarray*}
\int_{\R^3}\langle f(E),E\rangle\, dx = \int_{\R^3} |\curlop E|^2\, dx =6\int_{\R^3} F(E)\, dx.
\end{eqnarray*}
From (F6) we get
$$p\int_{\R^3} F(E)\, dx\leq 6\int_{\R^3} F(E)\, dx\leq q\int_{\R^3} F(E)\, dx.$$
Therefore
$\int_{\R^3} F(E)\, dx=0$ and $E=0$ a.e. on $\R^3$. Thus $u=0$ and we obtain a contradiction. If $V(x)=V_0<0$ is constant and $E=u+\nabla w$ is a classical solution to \eqref{eq} with $u\neq 0$ and  $q\leq 6$, then by Theorem \ref{ThPohozaev} 
\begin{eqnarray*}
(-V_0)|E|_2^2+\int_{\R^3}\langle f(E),E\rangle\, dx = \int_{\R^3} |\curlop E|^2\, dx =6\Big(\int_{\R^3} F(E)\, dx+\frac12(-V_0)|E|_2^2\Big)
\end{eqnarray*}
and, similarly as above, we get a contradiction.
\end{altproof}

{\bf Acknowledgements.} I would like to thank Thomas Bartsch for drawing my attention to semilinear Maxwell equations and for stimulating discussions on the topic. I am also indebted to the referee for many valuable comments.

\noindent {\sc Address of the author:}\\[1em]
\parbox{9cm}{Jaros\l aw Mederski\\
 Nicolaus Copernicus University \\
 Faculty of Mathematics and Computer Science\\
 ul.\ Chopina 12/18, 87-100 Toru\'n, Poland\\
 jmederski@mat.umk.pl\\
 }
 
\end{document}